\documentclass[a4paper,11pt]{article}
\usepackage[nointlimits]{amsmath}
\usepackage{amsfonts}
\usepackage{amssymb}
\usepackage{amsmath}
\usepackage{amsthm}
\usepackage{latexsym}
\usepackage{graphicx}
\usepackage{xcolor}
\usepackage{layout}
\usepackage{float}
\usepackage{hyperref}
\usepackage{tikz}
\usetikzlibrary{arrows.meta,calc,decorations.markings}
\usepackage[english]{babel}
\newtheorem{lemma1}     {Lemma}[section]
\newtheorem{teorema1}   [lemma1]{Theorem}
\newtheorem{prop1}      [lemma1]{Proposition}
\newtheorem{coroll1}    [lemma1]{Corollary}
\newtheorem{cong1}      [lemma1]{Conjecture}
\newtheorem{remark1}    [lemma1]{Remark}
\newtheorem{defin1}     [lemma1]{Definition}

\newenvironment{Lemma}[1][]
        {\begin{lemma1}[#1]\begin{samepage}}{\end{samepage}\end{lemma1}}
\newenvironment{Theorem}[1][]
        {\begin{teorema1}[#1]\begin{samepage}}{\end{samepage}\end{teorema1}}
\newenvironment{Proposition}[1][]
        {\begin{prop1}[#1]\begin{samepage}}{\end{samepage}\end{prop1}}

\newenvironment{Remark}[1][]
        {\begin{remark1}[#1]\begin{samepage}}{\end{samepage}\end{remark1}}
\newenvironment{Definition}[1][]
        {\begin{defin1}[#1]\begin{samepage}}{\end{samepage}\end{defin1}}

\numberwithin{equation}{section} 
\newtheorem{theorem}{Theorem}[section]

\newtheorem{lemma}[theorem]{Lemma}

\newtheorem{example1}   [lemma1]{Example}

\newcommand{\argf}{\Theta}
\newcommand{\lift}{\overline{\Theta}_\Ga}

\newcommand{\eps}{\epsilon}
\newcommand{\de}{\delta_\eps}
\newcommand{\dgae}{\dot\Ga_\eps}

\newcommand{\ex}{L}
\newcommand{\pq}{L}

\newcommand{\lgae}{\ell(\gae)}
\newcommand{\ga}{\gamma}
\newcommand{\Ga}{\Upsilon}

\newcommand{\gae}{\gamma_\eps}

\newcommand{\kae}{\kappa_{\gae}}
\newcommand{\sech}{\operatorname{sech}}
\newcommand{\R}       {\mathbb R}

\newcommand{\raccordo}{\varsigma_\eps}
\newcommand{\keycon}{\eta_\eps}
\newcommand{\recovery}{\gae}
\newcommand{\trec}{\tilde\gae}
\newcommand{\parameter}{x}
\newcommand{\radius}{\rho}
\newcommand{\prim}{\tau_\eps}

\newcommand{\nada}[1]{}

\title{Concentration effects and $\Gamma$-limit       
for the elastica functional for open and closed curves}

\author{
		Giovanni Bellettini\footnote{ 
		 Dipartimento di Scienze Matematiche, Informatiche e Fisiche,
via delle Scienze 206, 33100 Udine UD, Italy,
and International Centre for Theoretical Physics ICTP, Mathematics Section, 34151 Trieste, Italy.
		E-mail: giovanni.bellettini@uniud.it}
	\and
	Virginia Lorenzini\footnote{
		Dipartimento di Ingegneria dell'Informazione e Scienze Matematiche, Universit\`a di Siena, 53100 Siena, Italy.
		E-mail: virginia.lorenzini@student.unisi.it
	}
    \and
	Matteo Novaga \footnote{ 
		Dipartimento di Matematica, Universit\`a di Pisa, 56127 Pisa, Italy.
		E-mail: matteo.novaga@unipi.it}
    \and
	Riccardo Scala\footnote{ 
		Dipartimento di Ingegneria dell'Informazione e Scienze Matematiche, Universit\`a di Siena, 53100 Siena, Italy.
		E-mail: riccardo.scala@unisi.it}
}

\begin{document}
\maketitle
\begin{abstract}
\noindent
We study the $\Gamma$-convergence
of a class of elastica-type energies defined on immersed planar curves and depending on a  small positive parameter $\eps$.
As $\eps\to 0^+$, sequences with equibounded energy develop concentration phenomena in the curvature, leading to the emergence of singularities described by atomic measures. This naturally gives rise to a limiting framework in terms of pointed curves, consisting of a curve together with a measure encoding curvature concentration.
We characterize the first-order $\Gamma$-limit in two settings: for immersed open curves with fixed endpoints and boundary conditions on the tangents, and for immersed closed curves of prescribed length. In both cases, the limiting energy depends only on the number of concentration points and is expressed as a sum of contributions, each given by an integer multiple of $2\pi$.
A key feature of the problem is that the rescaled energies exhibit a structure closely related to one-dimensional Modica--Mortola type functionals.
\end{abstract}

\vspace{2mm}

\noindent {\bf AMS (MOS) 2020 Subject Clas\-si\-fi\-ca\-tion:}  49J45, 49Q20, 35B25

\section{Introduction}
In this paper, we investigate the $\Gamma$-limit
of elastica-type
functionals defined on immersed planar curves and depending on a small
parameter $\eps>0$.
This asymptotic analysis
is motivated by the fact that it captures information on the way curvature concentrates in small regions, which is not detected at the level of the zeroth-order limit. In this regime, the appropriate limiting objects are not described only by the limiting curve, but also by an additional measure recording the possible concentration of curvature. This naturally leads to the notion of \emph{pointed curves}, namely pairs formed by a curve $\ga$ and a compatible Radon measure $\mu$.

\medskip

We deal with two situations.
The first one concerns immersed plane curves $\ga\in H^2([0,\ell(\ga)];\mathbb R^2)$ joining two fixed points
$p,q\in\mathbb R^2$, $p,q\in\{y=0\}$, $p\neq q$ and with horizontal tangent at the endpoints. For this problem we consider the energy
$$
F_\eps(\ga)
:=
\int_0^{\ell(\ga)}(1+\eps \kappa_\ga^2)\,ds,
$$
where, as usual, $ds$ is the arclength element, $\kappa_\ga$ is the curvature of the curve, $\ell(\ga)$ its length and $\eps\in(0,1]$.
As $\eps\to 0^+$, minimizing sequences of $F_\eps$ approach the straight segment joining
$p$ and $q$, whose length is $|p-q|$. Our aim is to
study the asymptotic
behaviour of $F_\eps$, namely
\begin{equation}
\begin{aligned}\label{eq:G_eps}
G_\eps(\ga) :=
\frac{F_\eps(\ga)
- 
\vert p -q\vert}{\eps^{1/2}}
=
 \eps^{1/2}\int_0^{\ell(\ga)} \kappa_\ga^2~ds 
 + \frac{1}{\eps^{1/2}} 
\left(
\ell(\ga) - \vert p-q\vert\right)  
\end{aligned},
\end{equation}
for sequences with equibounded energy.
What happens is that, although the curves $\gae$ converge to the straight segment, the
curvature $\kae$ may concentrate in regions of size of order $\eps^{1/2}$,
producing a finite number of singularities, described by an atomic measure $\mu$ of the form 
\begin{equation}\label{eq:atomic_measure}
\mu=
2\pi
\sum_{j=1}^N c_j \delta_{s_j}, 
\quad
N \in \mathbb N, \ s_j \in [0,\ell(\ga)], 0\leq s_1<\cdots<s_N\leq\ell(\ga), c_j \in \mathbb Z.
\end{equation}
 The limiting object is thus given by the pair formed by the segment and $\mu$,
and is referred to as a \emph{pointed segment}.
Our first main result can be stated as follows (a more precise formulation
will be given in Theorem~\ref{thm:main_result}).

We let 
$$
\sigma:=\int_0^{2\pi}2\sqrt{1-\cos\phi} \,d\phi=8\sqrt{2}.
$$
\begin{Theorem}
 Let $(\gamma_\eps)$ be a sequence of immersed plane curves
$\gamma_\eps \in H^2([0,\ell(\gae)];\mathbb R^2)$ joining $p$ and $q$, with
horizontal tangent at the endpoints,
and assume that $G_\eps(\gamma_\eps)$ is uniformly bounded with respect to $\eps$. Then, up to subsequences,
$\gamma_\eps$ converges in $H^1$ to the straight segment joining $p$ and $q$, and
$\kappa_{\gamma_\eps}$ converges, in the flat norm, to an atomic measure $\mu$ of the form \eqref{eq:atomic_measure}. In this case, we say that $\big((\gamma_\eps,\kappa_{\gamma_\eps})\big)$ converges to the pointed segment $(\gamma,\mu)$.

Moreover, if $\big((\ga_\eps,\kappa_{\ga_\eps})\big)$ converges to $(\ga,\mu)$, then 
$$\liminf_{\eps\to 0^+} G_\eps(\gamma_\eps)
\geq \sigma \, G(\gamma,\mu),$$
where $$G(\gamma,\mu) = \sum_{j=1}^N |c_j|.$$

Conversely, for every pointed segment $(\gamma,\mu)$, there exists a sequence
$\big((\gamma_\eps,\kappa_{\gamma_\eps})\big)$ converging to $(\gamma,\mu)$ such that
$$
\lim_{\eps\to 0^+} G_\eps(\gamma_\eps)
= \sigma \, G(\gamma,\mu).
$$
\end{Theorem}

\medskip

The second situation we deal with concerns closed immersed plane curves $\gamma_\eps\in H^2([0,\ell];\mathbb{R}^2)$ of prescribed length $\ell$.
Given a fixed closed embedded curve $\gamma\in C^2([0,\ell];\mathbb{R}^2)$, we study
the energy
\begin{equation}\label{eq:mathcal_G_eps}
\mathcal{G}_\eps(\gamma_\eps)
:=
\eps^{1/2}\int_0^\ell \kappa_{\gamma_\eps}^2\,ds
+
\frac{1}{2\eps^{1/2}}\int_0^\ell
|\partial_s\gamma_\eps-\partial_s\gamma|^2\,ds.
\end{equation}
The second term (a sort of length excess) on the right hand side of \eqref{eq:G_eps} is now zero, and this induces a loss of compactness. For this reason, we add 
the second term in \eqref{eq:mathcal_G_eps}, which enforces the convergence of $\gamma_\eps$ to $\gamma$ in $H^1([0,\ell];\mathbb R^2)$,
and provides a quantitative control on the rate of convergence.
As in the first problem, sequences $\big((\gae,\kae)\big)$ with equibounded energy may develop concentration of curvature, and the limiting behavior is again described in terms
of \textit{pointed curves}.
Our second main result can be stated as follows (see Theorem \ref{thm:main_result_2} for a detailed formulation).
\begin{Theorem}
Let $(\gamma_\eps)$ be a sequence of closed immersed curves $\gamma_\eps \in H^2([0,\ell];\mathbb R^2)$ of length $\ell$
with $\mathcal{G}_\eps(\gamma_\eps)$ uniformly bounded with respect to $\eps$.
Then, up to subsequences, $\gae$ converges in $H^1$ to $\ga$, and  $\kappa_{\gamma_\eps}$ converges, in the flat norm,
to an atomic measure $\mu$ of the form \eqref{eq:atomic_measure}. In this case, we say that $\big((\gamma_\eps,\kappa_{\gamma_\eps})\big)$ converges to the pointed closed curve $(\gamma,\mu)$.

Moreover, if $\big((\ga_\eps,\kappa_{\ga_\eps})\big)$ converges to $(\ga,\mu)$, then 
$$
\liminf_{\eps\to 0^+} \mathcal{G}_\eps(\gamma_\eps)
\geq \sigma \, G(\gamma,\mu),
$$
where 
$$G(\gamma,\mu)= \sum_{j=1}^N |c_j|.$$ 

Conversely, for every pointed closed curve $(\gamma,\mu)$, there exists a sequence
$\big((\gamma_\eps,\kappa_{\gamma_\eps})\big)$ converging to $(\gamma,\mu)$ such that
$$
\lim_{\eps\to 0^+} \mathcal{G}_\eps(\gamma_\eps)
= \sigma \, G(\gamma,\mu).
$$
\end{Theorem}

\medskip

We observe that, at the scaling level considered here, no limiting quantity
depends on the distance between the singular points $s_j$. In other words, 
the asymptotic
energy detects the number of curvature concentration events, while their position
remains undetermined.
Capturing the mutual interactions between such concentrations 
would require a different 
choice of functionals;  however, this analysis appears to be
considerably more challenging, as even identifying the correct scaling seems a
nontrivial issue.

\medskip

It is important to note that the functionals $G_\eps$ and $\mathcal{G}_\eps$ are closely related, in structure, to one-dimensional Modica--Mortola type energies \cite{ModicaMortola}
(see \cite{Mi:20} for a related observation).
Indeed, when written in terms of the function $\theta:[0,\ell(\ga)]\to\mathbb R$,
representing a lifting of the tangent vector $\partial_s\gamma$
(i.e., $\partial_s\gamma=(\cos\theta,\sin\theta)$),
the functionals $G_\eps$ and $\mathcal{G}_\eps$ take the form of a gradient term plus a periodic potential
(see Remarks~\ref{rem:phase_transition} and~\ref{rmk:Modica-Mortola}).
The ideas introduced by Modica and Mortola are also useful in the present context, notably in the proof of the lower bound. However, some care is required, as our framework presents some differences
with respect to the classical phase transition setting: in particular, geometric
constraints are imposed on the curves, and the integration interval
$[0,\ell(\gamma)]$ is not fixed in the first problem.

\medskip

As for the proof strategy, the compactness argument is based on a careful analysis of the tangent field. More precisely, one isolates the regions (see \eqref{eq:E^eps_delta}) where the tangent vector deviates from the limiting direction and shows that each such region carries, in the flat limit, a curvature concentration with weight $\pm 2\pi$, while the curvature outside these regions is negligible in the flat norm (see Proposition \ref{lem:compactness_in_flat_norm}). This yields convergence, up to subsequences, to an atomic measure $\mu$ as in \eqref{eq:atomic_measure}.
This compactness statement identifies the correct limiting object for the
first-order analysis. The lower bound is obtained via the so-called
Modica--Mortola trick, while the upper bound is achieved through the
construction of suitable recovery sequences, in which the borderline elastica
plays a central role.
The Euler--Lagrange equation associated with $F_\eps$ reads
\begin{equation}\label{eq:Euler_Lagrange}
- \kappa + \eps(2 \kappa_{ss} + \kappa^3)=0.
\end{equation}
An arclength parametrized curve $\ga:\mathbb R\to \mathbb R^2$ is called planar elastica if its curvature satisfies \eqref{eq:Euler_Lagrange}.
Planar elasticae can be completely classified (essentially due to Euler)
and admit explicit parametrizations. Among them, the borderline elastica
is the only non-periodic case (see paragraph \ref{par:borderline_elastica}).

The construction of a recovery sequence $(\gae)$ in the case of immersed closed curves of
fixed length is more involved, since it must also preserve the length constraint
$\ell(\gamma_\eps)=\ell$.

\medskip

A related problem, exhibiting some analogies with our setting, has been studied by
Braides and Malchiodi in \cite{BrMa}. In their work, the authors consider functionals
defined on the boundaries of sets $E\subset\mathbb{R}^2$ of the form
$$
\int_{\partial E} \left( \eps\kappa^2 + \frac{1}{\eps}\psi(\nu_{E}) \right)\, d\mathcal{H}^1,
$$
where $\kappa(x)$ denotes the curvature of $\partial E$ at $x$, $\nu_{E}$ the outer unit normal vector of $E$ and $\psi:\mathbb{S}^1\to[0,+\infty)$ is a function with a finite number of zeros.
They compute the $\Gamma$-limit with respect to the $L^1$-convergence for this type of functionals and show that, if $(E_\eps)$ has uniformly
bounded energy, then the
corresponding boundaries converge, up to subsequences, to a
polygon. In the limit, the energy concentrates at the vertices, and the
limiting functional can be expressed as a sum over all vertices.
The techniques used in their analysis are closely related to the Modica--Mortola
approach, in particular for identifying the energy contribution of each singular point.
However, their setting is considerably different from ours. 
First, 
they consider boundaries of sets, whereas we deal with immersed curves. The reason for which we consider immersed curves is given by the fact that our analysis is, in some way, related to the elastic flow, which typically exhibits
self-intersections. 
Second, the notion of convergence in \cite{BrMa} is in $L^1$, while our
framework also encodes curvature concentrations, with
curves converging in $H^1$ and the curvature in the flat norm.
Third, in \cite{BrMa} a limit crystalline energy arises, which instead is not related to our problem.
On the other hand, both their problem and ours are geometric 
and the structure of the 
$\Gamma$-limit is closely related to the theory of phase transitions.

Our analysis can be put in the framework
of $\Gamma$-expansions: in this respect, we recall that 
first and second-order $\Gamma$-expansions for Modica--Mortola type functionals,
in their standard (non-geometric) formulation, have been studied, for instance, in the one-dimensional
case in \cite{BNN} and in higher dimensions in \cite{ABO, DMFL}.



\medskip

We finally emphasize that the analysis carried out in this paper is purely
static, and no evolution is considered.
Concerning evolutions, we recall that De Giorgi suggested in \cite{DG} to study
a fourth-order regularization of mean curvature flow based on the functional
$F_\eps$, possibly extended to arbitrary dimension and codimension.
In this setting, curves evolve according to the gradient flow of $F_\eps$,
which can be regarded as a higher-order perturbation of the curvature flow,
coinciding with it when $\eps=0$. It is known that, for every $\eps\in(0,1]$,
this evolution admits a unique smooth solution defined for all positive times.
In contrast, the immersed limit evolution develops singularities in finite time.
Such an approach has been proposed as a possible framework for defining
generalized solutions to the mean curvature flow, capable of describing the
evolution beyond singularities 
(see, for example, \cite{BMN, BLNS} in this direction).
We also mention that crystalline variants of the elastic flow have been studied
in the literature, see e.g. \cite{BKN}.

\medskip

The paper is organized as follows. In Section~\ref{sec:preliminaries} we collect
the basic notions on flat convergence of measures, $BV$
functions, and pointed curves.
In Section~\ref{sec:energy_on_X} we introduce the energy in the setting of open curves with boundary
conditions and study its asymptotic behaviour
in Section~\ref{sec:first_oder_exp_X}. More precisely, Section~\ref{sec:compactness} is devoted to compactness, Section~\ref{sec:lb}
to the $\Gamma$-liminf inequality, and Section~\ref{sec:up} to the $\Gamma$-limsup inequality.
Finally, in Section~\ref{sec:energy_on_Xl} we study the asymptotic problem for closed curves of fixed length and
establish the corresponding compactness, $\Gamma$-liminf, and recovery-sequence
results in Sections~\ref{sec:compactness_2}, \ref{sec:lb_2}, and \ref{sec:up_2}, respectively.

\section{Notation and preliminaries}\label{sec:preliminaries}
\nada{We recall that for a smooth closed plane curve
$$
\frac{1}{2\pi} 
\int_\gamma \kappa_{\gamma} = {\rm ind}(\gamma) \in \mathbb Z
$$
where ${\rm ind}(\gamma)$ is the winding number. 

\begin{Remark}\rm
For a map $\gamma \in H^2(0,\ell(\gamma))$ parametrized by arc-length, we have 
$\gamma_s \in H^1((0,\ell(\gamma), \mathbb S^1)$.
\textcolor{red}{This could be a relation with Ginzburg Landau?****}
\end{Remark}}

In this section, we collect the notation, definitions, and preliminary results used throughout the paper.

\subsection{Flat norm of Radon measures}
Given the interval $[0,\ex]$, we introduce the concept of flat norm of a Radon measure $\mu$, denoted by $\|\mu\|_{\text{flat},[0,L]}$, as
\begin{equation*}
\|\mu\|_{\text{flat},[0,\ex]} := \sup_{\substack{\varphi \in C^{0,1}([0,\ex]), \\ \|\varphi\|_{C^{0,1}([0,\ex])} \le 1}}
\int_{[0,\ex]} \varphi \, d\mu \
.
\end{equation*}
Here $\|\varphi\|_{C^{0,1}([0,\ex])}$ is given by
\begin{equation*}
\|\varphi\|_{C^{0,1}([0,\ex])} := \|\varphi\|_{L^{\infty}([0,\ex])} \;+\; \sup_{\substack{x,y \in [0,\ex] \\ x \neq y}} \frac{|\varphi(x) - \varphi(y)|}{|x - y|}.
\end{equation*}
\subsection{Functions of bounded variation}
We indicate by $I$ a bounded open interval of $\mathbb{R}$; if $x\in I$ we denote by $\delta_x$ the Dirac mass concentrated at $x$. 

If $\lambda$ is a scalar or vector-valued Radon measure, its total variation will be denoted by $|\lambda|$ and $\mathcal{M}_b(I)$ will be the space of Radon measures with bounded total variation.

We will recall the main properties of functions of
bounded variation, and we refer to \cite{AFP} for more details. 

The space $BV(I)$ is defined as the space of all functions $f\in L^1_{\rm loc}(I)$ whose distributional derivative $\partial_xf$ is a Radon measure with bounded total variation in $I$. We say that $f=(f_1,f_2):I\to\mathbb{R}^2$ belongs to $BV(I;\mathbb{R}^2)$ if $f_i\in BV(I)$ for $i=1,2$.

Given $f\in BV(I)$, we shall write
$$
\partial_xf=\partial_x^af \,dx+\partial_x^sf,
$$
where $\partial_x^af\in L^1(I)$ is the density of the absolutely continuous part of $\partial_xf$ with respect to the Lebesgue measure $dx$ on $I$, and $\partial_x^sf$ stands for the singular part. We shall use the same notation whenever $f\in BV(I;\mathbb{R}^2)$; obviously in this case $\partial_x^af\in L^1(I;\mathbb{R}^2)$.

If $f\in BV(I)$ we indicate by $J_f$ the jump set of $f$, and we set $f(x^-)={\rm ap}-\liminf_{y\to x} f(y)\leq f(x^+)={\rm ap}-\limsup_{y\to x} f(y)$. It is well known that $J_f=\{x\in I: f(x^-)<f(x^+)\}=\{x\in I:|\partial_xf|(\{x\})>0\}$. If $f=(f_1,f_2)\in BV(I;\mathbb{R}^2)$, by $J_f$ we mean $J_{f_1}\cup J_{f_2}$.

We also recall that functions in $BV(I)$ of one variable are bounded and admit traces at the endpoints of $I$.

We say that a sequence $(f_n)\subseteq BV(I;\mathbb{R}^2)$ converges to $f\in BV(I;\mathbb{R}^2)$ strictly in $BV(I;\mathbb{R}^2)$ if $f_n\to f$ in $L^1(I;\mathbb{R}^2)$ and $|\partial_xf_n|\to|\partial_xf|$ as $n\to+\infty$.

Analogous facts hold when $I$ is a bounded closed interval. 
\subsection{Pointed curves}
Let $p,q \in \mathbb{R}^2$, with  $p,q\in\{y=0\}$, $p \neq q$. In the following, it will be convenient to identify the interval $[p,q]$ with $[0,\ex]$, where $\ex:=|p-q|$.
\begin{Definition}
 We say that a curve $\Ga$ is of class $\mathcal{B}$ if  $\Ga\in H^1([0,\ex];\mathbb{R}^2)$ and $\partial_x{\Ga}\in BV([0,\ex];\mathbb{R}^2)$. We write $\Upsilon\in\mathcal{B}$.
\end{Definition}
We denote by $\displaystyle\ell(\Ga)=\int_{0}^\ex|\partial_x\Ga(x)| \,dx$ the length of $\Upsilon$.
\begin{Definition}\label{def:B_c}
We say that $\Ga$ is of class $\mathcal{B}_c$ if $\Ga\in\mathcal{B}$ and $|\partial_x\Ga(x)|={\rm c}= \frac{\ell(\Ga)}{\ex}$ for a.e. $x\in [0,\ex]$.   
\end{Definition}

Any curve $\Ga\in\mathcal{B}_c$ can be reparametrized by arc-length on $[0,\ell(\Ga)]$ so that its reparametrization $\ga$ satisfies $|\partial_s\ga(s)|=1$ for a.e. $s\in[0,\ell(\Ga)]$. We shall write $\ga\in\mathcal{B}_1$.

Notice that $x=s\frac{\ex}{\ell(\ga)}$ and 
$\partial_x=\frac{\ell(\ga)}{\ex}\,\partial_s.$

\medskip

Following \cite[Lemma 3.1]{BellettiniPaolini1995}, we now define the angle formed by the tangent vector to a curve of class $\mathcal{B}_1$ and call it a minimal argument (or minimal lifting) of $\partial_s\ga$.
\begin{Lemma}[minimal lifting]\label{lem:argument_fun}
Let $\gamma\in\mathcal{B}_1$. Then there exists a function $\argf=\argf_\ga:(0,\ell(\ga))\to\mathbb{R}$ having the following properties:
\begin{itemize}
    \item [\textnormal{(i)}] $\argf\in BV((0,\ell(\ga)))$ and $\partial_s\gamma(s)=(\cos\argf(s),\sin\argf(s))$ for almost every $s\in(0,\ell(\ga))$;
    \item [\textnormal{(ii)}] $J_{\argf}=J_{\partial_s\ga}$;
    \item [\textnormal{(iii)}] $-\pi<\argf(s^+)-\argf(s^-)\leq\pi$ for every $s\in(0,\ell(\ga))$.
\end{itemize}
\end{Lemma}
The choice of $\argf$ is unique modulo $2\pi$. 

\medskip

\textbf{Notation:} If $\Ga\in \mathcal B_c$ we will always denote by $\gamma:=\gamma(\Ga)\in \mathcal{B}_1$ its arc-length parametrization. If $\Theta=\Theta_\gamma$ is a minimal lifting associate to $\gamma$, we denote by $\Theta_\Ga:(0,L)\rightarrow \R$ the function $\Theta_\Ga(x)=\Theta(\frac{L}{\ell(\gamma)}x)$.
 \medskip

If $\Ga\in \mathcal B_c\cap H^2([0,L];\mathbb{R}^2)$ then $\ga\in \mathcal{B}_1\cap H^2([0,\ell(\ga)];\mathbb{R}^2)$ and there exists a lifting $\theta_\ga$ of $\partial_s\ga$ of class $H^1$ such that $\partial_s\gamma=(\cos\theta_\ga(s),\sin\theta_\ga(s))$ for all $s\in[0,\ell(\ga)]$; thus we introduce the (signed) curvature of $\ga$ as
\begin{equation}\label{eq:partial_stheta}
    \kappa_{\ga}:=\partial_s\theta_{\ga},
\end{equation}
whose absolute value coincides with the scalar curvature of $\ga$,
since $\partial^2_{s}\ga=\partial_s\theta_\ga(-\sin\theta_\ga,\cos\theta_\ga)=\kappa_\ga(-\sin\theta_\ga,\cos\theta_\ga)$. 
We denote by $\kappa_\Ga$ the (signed) curvature of $\Ga$ defined as 
\begin{equation}\label{eq:curvature}
\kappa_{\Ga}(\parameter)=\kappa_\ga\Big(\frac{\ell(\ga)}{\ex}x\Big)=\kappa_\ga(s)
\quad \forall x\in [0,\ex].
\end{equation}

\medskip

If $\ga\in \mathcal{B}_1$ 
we denote by $\partial_{ss}\ga$ the distributional derivative of $\partial_s\ga$; we write
$\partial_{ss}\ga=\partial_{ss}^a\ga \,ds+\partial_{ss}^s \gamma$.
Combining the chain rule for $BV$ functions with (i) of Lemma \ref{lem:argument_fun} and the uniqueness of the Lebesgue decomposition for a measure, one readily obtains that
\begin{equation}\label{eq:ass_con}
    \partial_{ss}^a\ga=(-\sin\argf,\cos\argf)\partial_s^a\argf \quad \text{a.e. in } [0,\ell(\ga)]
\end{equation}
and, if $B\subseteq [0,\ell(\ga)]$ is a Borel set, then 
\begin{equation}
\begin{aligned}\label{eq:partial_ss_sing}
    \partial_{ss}^s\ga(B)=&\int_{B\cap([0,\ell(\ga)]\setminus J_\argf)} 
    (-\sin\argf,\cos\argf) \,d\partial_s^s\argf 
    \\
    &
    +
    \sum_{t\in B\cap J_{\argf}} (\cos\argf(t^+)-\cos\argf(t^-),\sin\argf(t^+)-\sin\argf(t^-))\delta_t.
\end{aligned}
\end{equation}
Therefore \eqref{eq:ass_con} implies
$ |\partial_{ss}^a\ga|=|\partial_s^a\argf|$ a.e. in $[0,\ell(\ga)]$.
Note also that
$
|\partial_{ss}^s\gamma|(\{s\})<|\partial_s^s\argf|(\{s\})$ for all $s\in J_\argf$.
Indeed, for any $s\in J_\argf$, using \eqref{eq:partial_ss_sing} and $2(1-\cos\phi)=4\sin^2(\phi/2)$, we have 
\begin{align*}
    (|\partial^s_{ss}\ga|(\{s\}))^2
    &=
    |(\cos\argf(s^+)-\cos\argf(s^-),\sin\argf(s^+)-\sin\argf(s^-)|^2
    \\
    &=4\sin^2((\argf(s^+)-\argf(s^-))/2)
    \leq (\argf(s^+)-\argf(s^-))^2=(|\partial_s^s\argf|(\{s\}))^2.
\end{align*}
\begin{Definition}
We let $M_{{\rm fin}, \mathbb Z}
([0,\ex])$ be the class of atomic measures $\omega$ on $[0,\ex]$
of the form 
$$
\omega=
2\pi
\sum_{j=1}^N c_j \delta_{x_j}, 
\qquad N \in \mathbb N, \ x_j \in [0,\ex], c_j \in \mathbb Z,
$$
with $0\leq x_1<\cdots<x_N\leq \ex$.
\end{Definition}
We observe that the total variation of $\frac{\omega}{2\pi}$ is $\sum_{j=1}^N |c_j|$.
\begin{Definition}
 Let $\Ga\in\mathcal{B}_c$. We define the set of measures compatible with $\Ga$ as
\begin{equation*}
\mathcal{K}(\Ga) := 
\left\{
\mu \in \mathcal{M}_b([0,\ex]) :
\mu = \frac{\ex}{\ell(\Ga)}\partial_x \argf_{\Ga} +\omega \text{ with } \omega\in M_{{\rm fin}, \mathbb Z}([0,\ex])
\right\}.
\end{equation*}
\end{Definition}
\begin{Definition}
A pair $(\Ga,\mu)$ with $\Ga \in \mathcal{B}_c$ and 
$\mu \in \mathcal{K}(\Ga)$ is called a pointed curve.
\end{Definition}
We sometimes will refer to $\omega$ as the singularities of $\Ga$.
If $\Ga\in \mathcal{B}_c\cap H^2([0,\ex];\R^2)$ then $(\Ga,\kappa_\Ga)$ is a pointed curve, without singularities.
\begin{Lemma}
 Given a pointed curve  $(\Ga,\mu)$ there exists a function  
$\lift:[0,\ex]\to\mathbb{R}$  such that 
\begin{itemize}
    \item[\textnormal{(i)}] $\lift\in BV([0,\ex])$ and $\partial_x\Ga(x)=\frac{\ell(\Ga)}{\ex}(\cos\lift(x),\sin\lift(x))$ for almost every $x\in[0,\ex]$;
    \item [\textnormal{(ii)}]
    if $\mu = \frac{\ex}{\ell(\Ga)}\partial_x \argf_{\Ga} +\omega$ with $\omega\in M_{{\rm fin}, \mathbb Z}([0,\ex]) $, then
    \begin{equation}\label{eq:partial_slift} 
    \partial_x \lift = \partial_x\argf_\Ga+\omega.
\end{equation}
\end{itemize}
Moreover, the choice of $\lift$ is unique modulo $2\pi$. 
\end{Lemma}
\begin{proof}
Let $\argf_\Ga$ be the minimal lifting of $\partial_x\Ga$, so that $\argf_\Ga\in BV([0,\ex])$ and $\partial_x\Ga(x)=\frac{\ell(\Ga)}{\ex}
\bigl(\cos\argf_\Ga(x),\sin\argf_\Ga(x)\bigr)$ for a.e. $x\in[0,\ex]$.
If $\omega=
2\pi
\sum_{j=1}^N c_j \delta_{x_j}, 
\  N \in \mathbb N, \ x_j \in [0,\ex], 0\leq x_1<\cdots<x_N\leq \ex , c_j \in \mathbb Z,$ then we define
$$\lift(x):=\argf_\Ga(x)+2\pi\sum_{j=1}^N c_j\,\chi_{(x_j,\ex)}(x),$$
where $\chi_{(x_j,\ex)}$ denotes the characteristic function of the interval $(x_j,\ex)$. By construction, (i) follows. Moreover, since $\partial_x\chi_{(x_j,\ex)}=\delta_{x_j}$ in the sense of distributions, we obtain
$\partial_x\lift
=
\partial_x\argf_\Ga
+
2\pi\sum_{j=1}^N c_j\,\delta_{x_j}
=
\partial_x\argf_\Ga+\omega.$ This proves \eqref{eq:partial_slift}.
Finally, uniqueness follows from \eqref{eq:partial_slift}. Indeed it yields
    \begin{align*}
        \lift(x)=\lift(0)+\int_0^x \partial_t\lift(t) \,dt,
    \end{align*}
where $\lift(0)$ is so that $\partial_x\Ga(0)=\frac{\ell(\Ga)}{\ex}(\cos\lift(0),\sin\lift(0))$. Hence $\lift$ is uniquely determined up to the choice of $\lift(0)$ modulo $2\pi$.
\end{proof}
\begin{Definition}\label{def:conv_B}
  We say that a sequence $(\Ga_n,\mu_n)$ of pointed curves converges to a pointed curve $(\Ga,\mu)$ if $(\Ga_n)$ converges to $\Ga$ in $H^1([0,\ex];\mathbb{R}^2)$ and $(\mu_n)$ converges to $\mu$ in the flat norm.
\end{Definition}
\section{The energy functionals in case of open curves}\label{sec:energy_on_X}
Let $p,q \in \mathbb{R}^2$, with  $p,q\in\{y=0\}$, $p \neq q$.
Let $\mathcal{T}=\{(\Ga,\mu): \Ga\in\mathcal{B}_c,
\mu\in\mathcal{K}(\Ga) \}$.
Define 
\begin{align*}
X: = \{(\Ga,\mu)\in \mathcal{T}
&:\Ga \in H^2([0,\ex]; \R^2),
|\partial_x\Ga|=\text{c},
\Ga(0)=p,
\Ga(\ex)=q,
\\
&\quad
\dot\Ga_2(0)=0,
\dot\Ga_2(\ex)=0, 
\mu=\kappa_{\Ga}
\}.
\end{align*}

Let $F_\eps:\mathcal{T}\to [0,+\infty]$ be defined as
\begin{equation}\label{eq:F_eps}
F_\eps(\Ga,\mu)
=\begin{cases}
 \ell(\Ga)+\eps \frac{\ell(\Ga)}{\ex}\displaystyle\int_0^\ex \kappa_\Ga^2 \,dx &
\text{ if }
(\Ga,\mu) \in X, 
\\
+\infty & \text{ if } (\Ga,\mu)\in\mathcal{T}\setminus X. 
\end{cases}
\end{equation}
\begin{Definition}
We denote by
$\rm Dom_G$ 
the set of all pairs $(\Ga,\mu) \in\mathcal{T}$, 
$\Ga(x)=p+\frac{x}{\ex}(q-p) \text{ for all }x\in[0,\ex]$,
$\mu\in M_{{\rm fin}, \mathbb Z}
([0,\ex])$.
\end{Definition}
An element of $\mathrm{Dom}_G$ will be called a pointed segment.

\nada{\begin{Definition}
We set
$\overline{{\rm Dom}_D}$ 
the set of all pairs $(\gamma, \mu)$
such that
\begin{itemize}
    \item [-] $\gamma \in \mathcal{B}$, $\gamma(0)=p$, $\gamma(\ell)=q$, $\dot\gamma_2(0)=0, \dot\gamma_2(\ell)=0$,
    \item [-] $\mu \in M_{{\rm fin}, \mathbb Z}(0,\ell)$ is so that $\mu=\dot\theta$, where $\theta$ is an argument of $\dot\ga$.
\end{itemize} 
\end{Definition}}


\begin{Definition}
We define $G:\mathcal{T} \to
[0, +\infty]$ 
the functional 
\begin{equation}\label{eq:G}
G(\Ga,\mu) := \begin{cases}
\sum_{j=1}^N \vert c_j\vert  & {\rm if }~ 
(\Ga,\mu) \in {\rm Dom_G},
\\
+ \infty & {\rm if }~  \Ga\in\mathcal{T}\setminus{\rm Dom_G}.
\end{cases}
\end{equation}
\end{Definition}
\begin{Definition}
We define $G_\eps:X\to[0,+\infty]$ as
$$
\begin{aligned}
G_\eps(\Ga,\kappa_\Ga) := &
\frac{F_\eps(\Ga)
- 
\min_{\Ga\in X} \ell(\Ga)}{\eps^{1/2}}
=
\frac{F_\eps(\Ga)
- 
\vert p -q\vert}{\eps^{1/2}}
\\
=&
\frac{\ell(\Ga)}{\ex}
\int_0^L \eps^{1/2} \kappa_\Ga^2~dx + \frac{1}{\eps^{1/2}} 
\left(
\ell(\Ga) - \vert p-q\vert\right).
\end{aligned}
$$
%
\end{Definition}
In what follows, for $\Ga\in H^2([0,\ex]; \R^2)$ 
we often write $G_\eps(\Ga)$ in place of $G_\eps(\Ga,\kappa_\Ga)$.
\section{\texorpdfstring{$\Gamma$-limit for open curves}{Gamma-limit for open curves}}\label{sec:first_oder_exp_X}

\nada{
A ``static'' solution of \eqref{eq:Euler_Lagrange} satisfies
$\kappa_{ss}=0$, so that 
$\kappa^2 \simeq \frac{1}{\eps}$, i.e.,
$$
\kappa 
\simeq \frac{1}{\eps^{1/2}}
$$
a circle of radius $\sqrt{\eps}$, for which $\eps W \simeq \sqrt{\eps} \simeq \ell$. 
This suggests to rescale dividing by $\eps^{1/2}$.}

\begin{Remark}\label{rem:phase_transition}\rm
 As already noted in~\cite{Mi:20}, $G_\eps(\gamma)$ can be interpreted as a one--dimensional Modica--Mortola type phase transition energy. \nada{whose best known version is 
\begin{equation*}
{\rm MM}_\eps(u)
= \int_\Omega \eps^{1/2}|\nabla u|^2
+
\int_\Omega \frac{1}{\eps^{1/2}} W(u),
\end{equation*}
where $\Omega\subset \mathbb{R}^n$ is an open set and $W$ a double-well potential.
Indeed, $G_\eps$ can be expressed in terms of the tangential angle function
$\theta\colon[0,\ell(\gamma)]\to\mathbb{R}$ (defined by $\partial_s\gamma=(\cos\theta,\sin\theta)$) as}
Indeed
\begin{equation*}
G_\eps(\gamma)
=
\int_0^{\ell(\gamma)} \eps^{1/2}(\partial_s\theta)^2\,ds
+
\int_0^{\ell(\gamma)} \frac{1}{\eps^{1/2}} (1-\cos\theta)\,ds, 
\end{equation*}
where we used \eqref{eq:partial_stheta} and the fact that $$\displaystyle\int_0^{\ell(\gamma)}  \cos\theta(s)\,ds=\int_0^{\ell(\ga)}\partial_s\ga_1(s) \, ds=q_1-p_1=|p-q|.$$
 Hence $G_\eps$ resembles a one--dimensional phase transition energy
with $2\pi$-periodic potential
$W(\theta)=1-\cos\theta \; (=2\sin^2(\theta/2)).$

Despite these similarities, our problem differs from the classical phase
transition framework in several respects. Most notably, the integration interval
$[0,\ell(\gamma)]$ is not fixed, and further constraints are present due to the boundary conditions, namely
\begin{equation}\label{eq:condizioni_bordo}
    \theta(0)=0, \quad \theta(\ell(\ga))=2\pi\mathbb Z, \quad\int_0^{\ell(\ga)} \cos\theta(s) \,ds=\ex, \quad \int_0^{\ell(\ga)}\sin\theta(s) \,ds=0.
\end{equation}
\end{Remark}

\medskip

We are interested in sequences $\big((\gamma_\eps,\kae)\big)$ such that there exists $C>0$ for which  
\begin{equation}\label{eq:uniform_bound}
G_\eps(\gamma_\eps) \leq C, \quad \eps\in(0,1].
\end{equation}
%


Notice that 
\eqref{eq:uniform_bound} implies $\ell(\gamma_\eps) \to \vert
p-q\vert$, and 
$$
\sqrt{\eps}
\int_{\gamma_\eps} \kappa^2_{\gamma_\eps} 
\leq C.
$$
Thus, intuitively, $\gamma_\eps$ is allowed to form 
a finite number (depending on $C$) of 
small loops, such as exact circles for instance,
 of radius $\sqrt{\eps}$, still keeping a uniform 
bound on their elastic energy.

\medskip

We now state the first main result of the paper.

\begin{Theorem}[Compactness and $\Gamma$-convergence]\label{thm:main_result}
We have:
\begin{itemize}
\item[\textnormal{(i)}] Compactness and $\Gamma$-liminf inequality: if $\big((\Ga_\eps,\kappa_{\Ga_\eps})\big)\subset X$ is a sequence such that 
\eqref{eq:uniform_bound} holds, then,
up to a subsequence, there exist a pointed segment $(\Ga,\mu)\in \rm Dom_G$ such that $\big((\Ga_\eps,\kappa_{\Ga_\eps})\big)$ converges to $(\Ga,\mu)$. Furthermore, if $\big((\Ga_\eps,\kappa_{\Ga_\eps})\big)$ converges to $(\Ga,\mu)$ then 
%
$$
\liminf_{\eps\to 0^+} G_\eps(\Ga_\eps)
\geq 
\sigma G(\Ga,\mu).
$$
\item[\textnormal{(ii)}] $\Gamma$-limsup inequality: 
for every $(\Ga,\mu)\in \rm Dom_G$
there exists a sequence $\big((\Ga_\eps,\kappa_{\Ga_\eps})\big)\subset X$ converging to $(\Ga,\mu)$ such that
$$
\lim_{\eps\to 0^+} G_\eps(\Ga_\eps)
=
\sigma G(\Ga,\mu).
$$
\end{itemize}

\nada{
Oppure: Let $(\gamma_\eps)$ converges in ... to $\gamma \in ...$.
Then 
$$
F_\eps(\gamma_\eps) \geq \ell(\gamma) + ... + o() ..
$$
Let $\gamma \in ...$. Then there exists a sequence
$(\gamma_\eps)$ converging to $\gamma$ in ... such that 
$$
F_\eps(\gamma_\eps) \leq \ell(\gamma) + ... + o() ..
$$
}
\end{Theorem}




\medskip

The proof of Theorem \ref{thm:main_result} is split across Sections \ref{sec:compactness}, \ref{sec:lb}, and \ref{sec:up}.
\subsection{The compactness result}\label{sec:compactness}

\nada{
On the rescaling $\eps^{1/2}$ we observe the following:
\begin{lemma}
If $\delta_\eps>0$ is such that $\frac{\eps^{1/2}}{\delta_\eps}\rightarrow 0$ as $\eps\rightarrow0$, the following $\Gamma$-limit holds
$$\frac{F_\eps(\gamma)
	- 
	\vert p -q\vert}{\delta_\eps}\rightarrow 0,$$
in some topology...
\end{lemma}

Assume $\gamma_\eps$ are curves so that 
$${F_\eps(\gamma)	-	\vert p -q\vert}=\lgae-|p-q|+\int_{\gamma_\eps}\eps\kae^2\leq C\delta_\eps$$
so that  
$$\lgae-|p-q|\leq C\delta_\eps,\qquad \qquad \int_{\gamma_\eps}\eps^{1/2}\kae^2\leq C\frac{\delta_\eps}{\eps^{1/2}} $$
Any ricciolino $c$ of radius $\rho$ pays $C\rho$ in length and $\int_c \eps^{1/2}k^2=\frac{\eps^{1/2}}{\rho}$. Then we can approximate the segment $\overline{pq}$ with curves that are almost equal to the segment and with the addition of an arbitrary large number $n$ of ricciolini, provided that $n\rho\rightarrow 0$ and $\frac{n\eps^{1/2}}{\rho}\rightarrow 0$. THis is ensured if $\eps^{1/2}<<\rho<<1$. At the same time to guaranteed the bounds above we want $n\rho<<\delta_\eps$, and $ \frac{n\eps^{1/2}}{\rho}<<\frac{\delta_\eps}{\eps^{1/2}}$. Hence 
$$\eps^{1/2}=\max\{\frac{\eps}{\delta_\eps},\eps^{1/2}\}<<\rho<<\delta_\eps$$
and such a $\rho$ can be always chosen thanks to $\frac{\eps^{1/2}}{\delta_\eps}\rightarrow 0$. For this $\rho$ the limit of the energies $F_\eps(\gamma_\eps)$ is zero, and we can approximate any atomic measure with the curvatures of $\gamma_\eps$.}
In this section we prove the compactness part of Theorem~\ref{thm:main_result}. To this purpose, 
assume that for some $C>0$
\begin{align}\label{bounds}
\lgae-|p-q|\leq C{\eps^{1/2}},\qquad \qquad \int_{\gamma_\eps}\eps^{1/2}\kae^2 \ d s\leq C.
\end{align}

Let $\theta_\eps$ denote the tangential angle function so that $\dgae(\parameter)=\frac{\lgae}{\pq}e^{i\theta_\eps(\parameter)}$. From \eqref{eq:curvature} and \eqref{eq:partial_stheta} we deduce that 
\begin{equation}\label{eq:scal_k}
    \kappa_{\Ga_\eps}(\parameter)=\frac{\pq}{\lgae}\dot\theta_\eps(\parameter).
\end{equation}
\begin{Lemma}\label{lem:H1}
    Assume that \eqref{bounds} holds. Then, up to a not relabelled subsequence, there exists $\Ga\in W^{1,\infty}([0,\pq];\R^2)$ such that 
    \begin{align*}
	\Ga_\eps\rightarrow \Ga\qquad \text{ in }H^1([0,\pq];\R^2)\text{ and weakly star in }W^{1,\infty}([0,\pq];\R^2).
\end{align*} 
Moreover there exists a constant $C>0$ such that 
\begin{equation}\label{eq:rate_conv}
    \|\dot\Ga_\eps-\dot\Ga\|^2_{L^2}\leq C\eps^{1/2}.
\end{equation}
\end{Lemma}
\begin{proof}
The first condition in \eqref{bounds} implies that $\lgae \to \pq$ as $\eps \to 0^+$. Consequently,
\begin{equation*}
|\dot\Ga_\eps(\parameter)| = \frac{\lgae}{\pq} \leq 1 + \frac{C}{\pq}\,\eps^{1/2} \leq C_0
\quad\text{for all } \parameter \in [0,\pq],
\end{equation*}
so the family $\{\Ga_\eps\}$ is uniformly Lipschitz on $[0,\pq]$. Since each $\Ga_\eps$ connects the fixed endpoints $p$ and $q$, the curves are uniformly bounded in $L^\infty([0,\pq];\mathbb{R}^2)$. Hence,
\begin{equation*}
\sup_{\eps\in(0,1]} \|\Ga_\eps\|_{W^{1,\infty}(0,\pq)} < +\infty.
\end{equation*}

We deduce that there exists a not relabelled subsequence and a limit curve $\Ga \in W^{1,\infty}([0,\pq];\mathbb{R}^2)$ such that
\begin{equation*}
\Ga_\eps \rightharpoonup^\ast \Ga \quad \text{in } W^{1,\infty}([0,\pq];\mathbb{R}^2).
\end{equation*}

Moreover, the total variation of each $\Ga_\eps$ satisfies
\begin{equation*}
|\dot\Ga_\eps|([0,\pq]) = \int_0^\pq |\dot\Ga_\eps|\,d\parameter = \lgae\longrightarrow \pq,
\end{equation*}
from which we deduce that $|\dot\Ga|([0,\pq]) = \pq$. This equality implies that $\Ga$ is the straight line segment from $p$ to $q$, and the convergence $\Ga_\eps \to \Ga$ is strict in $BV([0,\pq];\mathbb{R}^2)$.
Therefore, up to the extraction of a subsequence, we have 
 \begin{align}\label{conv1}
	\Ga_\eps\rightarrow \Ga\qquad \text{ strictly in } BV([0,\pq];\R^2)\text{ and weakly star in }W^{1,\infty}([0,\pq];\R^2).
\end{align} 
Notice also that, since $$\frac{\lgae}{\pq}\to 1,$$ the curve $\Ga$ can be parametrized with speed $|\dot\Ga|=1$ on $[0,\pq]$.

We now show that the convergence in \eqref{conv1} is strong in $H^1([0,\pq];\mathbb{R}^2)$. More precisely that \eqref{eq:rate_conv} holds.
First we observe that from \eqref{bounds} we can find a positive constant $C$ such that, for $\eps>0$ small enough,
$$\int_0^\pq|\dot\Ga_\eps|^2d\parameter-\int_0^\pq|\dot\Ga|^2d\parameter=\frac{\ell^2(\gae)}{\pq}-\pq=\frac{1}{\pq}(\ell(\gae)-\pq)(\lgae+\pq)\leq C\eps^{1/2}.$$
So we can write
\begin{align}\label{eq:stima}
C\eps^{1/2}\geq \int_0^\pq\langle\dot\Ga_\eps-\dot\Ga,\dot\Ga_\eps+\dot\Ga \rangle d\parameter=\int_0^\pq|\dot\Ga_\eps-\dot\Ga|^2d\parameter+2\int_0^\pq\langle\dot\Ga_\eps-\dot\Ga,\dot\Ga\rangle d\parameter.
\end{align}
Recalling that $$\dot\Ga=\frac{q-p}{\pq}=:v$$ is constant and that $\displaystyle\int_0^\pq\dot\Ga_\eps \, d\parameter=q-p$, it follows that 
\begin{equation*}
\int_0^\pq\langle\dot\Ga_\eps-\dot\Ga,\dot\Ga\rangle d\parameter=\frac{(p-q)}{\pq} \cdot\int_0^\ex \dot\Ga_\eps\,d\parameter-\frac{|p-q|^2}{\ex^2}\ex=\frac{|p-q|^2}{L}-\frac{|p-q|^2}{L}=0.
\end{equation*}
Hence the thesis follows from \eqref{eq:stima}.
\end{proof}

We now fix $\radius\in (0,\frac12)$ and consider the set
\begin{align}\label{eq:E^eps_delta}
E_\radius^\eps:=\{\parameter\in [0,\ex]:|\dot\Ga_\eps(\parameter)-\dot\Ga(\parameter)|=|\dot\Ga_\eps(\parameter)-v|>\radius\}.
\end{align}
By the estimate in the previous lemma we infer
\begin{align}\label{eq:grandezza_E}
|E_\radius^\eps|\radius^2
\leq \int_{E^\eps_\radius} |\dot\Ga_\eps(\parameter)-\dot\Ga(\parameter)|^2d\parameter
\leq
\int_0^\pq|\dot\Ga_\eps(\parameter)-\dot\Ga(\parameter)|^2d\parameter\leq C\eps^{1/2},
\end{align}
and so, by the Cauchy-Schwarz inequality and the second bound in \eqref{bounds} 
\begin{align}\label{eq:kx_uniformly}
\int_{E_\radius^\eps}|\kappa_{\Ga_\eps}|d\parameter\leq\left(\int_{E_\radius^\eps}
\eps^{1/2}\kappa_{\Ga_\eps}^2d\parameter\right)^{1/2}\left(\int_{E_\radius^\eps}\frac{1}{\eps^{1/2}}d\parameter\right)^{1/2}\leq \frac{C}{\radius}.
\end{align}
Therefore we find that $$\omega^\radius:=\kappa_{\Ga_\eps}\chi_{E_\radius^\eps}$$ is uniformly bounded in the space of Radon measures with respect to $\eps$, and hence, up to a subsequence, it converges to some measure as $\eps\to 0^+$.

\medskip
\begin{Proposition}\label{lem:compactness_in_flat_norm}
Assume that \eqref{bounds} holds. Then there exist a subsequence $\eps_k$ and $\omega\in M_{{\rm fin}, \mathbb Z}
([0,\ex])$ such that 
\begin{equation}\label{eq:compactness_in_flat_norm}
   \lim_{k\to \infty} \|\kappa_{\Ga_{\eps_k}} -\omega\|_{\mathrm{flat},[0,\ex]} = 0.
\end{equation}
\end{Proposition}
\nada{Then there exist an integer $N\geq 0$, points
$0\leq x_0\leq x_1\leq \dots\leq x_N\leq\pq$ and values $\alpha_j\in \{-1,+1\}$ such that   
\begin{equation}\label{eq:compactness_in_flat_norm}
   \lim_{\eps\to 0^+} \|\kappa_{\Ga_\eps} -2\pi\sum_{j=1}^N  \alpha_j\delta_{x_j}\|_{\mathrm{flat},[0,\ex]} = 0.
\end{equation}
Notice carefully that the points $x_j$ above are not necessarily distinct; 
therefore they could give rise to integer 
multiplicity Dirac deltas at the singular points.}
\begin{proof} Let $\radius\in(0,1/2)$ be fixed and 
$\eps>0$ be small enough so that the set 
\begin{equation}\label{eq:C_rho}
C_{\radius,\eps}
:=\partial B_{\frac{\lgae}{\pq}}(0)\setminus \overline{B_{\radius}(v)}=\Big\{y\in \R^2:|y|=\frac{\lgae}{\pq},\;|y-v|>\radius\Big\}
\end{equation}
	is an arc relatively open in $\partial B_{\frac{\lgae}{\pq}}(0)$, where we recall that $v=\frac{q-p}{\pq}$. Let us 
denote by 
\begin{equation}\label{eq:R_S}
R_{\eps,\radius}
\quad {\rm and} \quad S_{\eps,\radius}
\end{equation}
 the two endpoints of $C_{\radius,\eps}$, namely the two intersection points between $\partial B_{\frac{\lgae}{\pq}}(0)$ and $\partial B_{\radius}(v)$, taken in such a way that orienting  $C_{\radius,\eps}$ in counterclockwise order, $R_{\eps,\radius}$ and $S_{\eps,\radius}$ are the starting and ending points of the arc, respectively. 
	
	 On the set $E_\radius^\eps$ in \eqref{eq:E^eps_delta}, the map $\dot\Ga_{\eps}$ takes values in $C_{\radius,\eps}$. Moreover, using that $\dot\Ga_\eps(0)=\dot\Ga_\eps(\ex)=\frac{\ell(\gae)}{\ex}v$,
     by continuity of $\dot\Ga_\eps$ it turns out that $E_\radius^\eps=\dot\Ga_\eps^{-1}(C_{\radius,\eps})$ is open; hence there exist at most countably many mutually disjoint intervals $(a^i_{\eps,\radius},b_{\eps,\radius}^i)\subseteq (0,\ex)$ such that 
	$$E_\radius^\eps=\cup_{i=1}^\infty (a^i_{\eps,\radius},b_{\eps,\radius}^i).$$

\medskip

\begin{figure}[htbb]
\begin{center}
\begin{tikzpicture}[>=Latex, line cap=round, line join=round]


\draw[line width=0.9pt] (0,0) -- (6,0);
\draw[line width=0.9pt] (0,0.18) -- (0,-0.18);
\draw[line width=0.9pt] (6,0.18) -- (6,-0.18);

\node[below=8pt] at (0,0) {$0$};
\node[below=8pt] at (6,0) {$\pq$};

\def\aone{1.6}
\def\bone{2.4}
\def\atwo{2.435}
\def\btwo{3.235}

\node at (\aone,0) {$($};
\node at (\bone,0) {$)$};
\node at (\atwo,0) {$($};
\node at (\btwo,0) {$)$};

\node[below=8pt] at (\aone,0.2) {$a_{\eps,\radius}^1$};
\node[below=8pt] at (\bone,0.2) {$b_{\eps,\radius}^1$};

\node[above=8pt] at (\atwo+0.15,-0.2) {$a_{\eps,\radius}^2$};
\node[above=8pt] at (\btwo+0.15,-0.2) {$b_{\eps,\radius}^2$};

\draw[->,line width=1pt]
(2,0.5) to[out=30,in=170] (9.5,2);

\node at (6,2.4) {$\dot\Upsilon_\eps$};

\begin{scope}[shift={(11.5,0)}]

\draw[->] (-2.8,0) -- (3.8,0);
\draw[->] (0,-2.5) -- (0,2.5);

\fill (0,0) circle (1.2pt);
\node[below left] at (0,0) {$0$};

\def\R{2.0}
\def\vx{1.45}
\def\r{0.78}

\draw (0,0) circle (\R);
\node[rotate=300] at ({40:\R}) {\tiny$\blacktriangleleft$};
\node[rotate=300] at ({210:\R}) {\tiny$\blacktriangleright$};
\node at ({210:\R+0.4}) {$C_{\radius,\eps}$};

\fill (\vx,0) circle (1.2pt);
\node[below=4pt] at (\vx,0) {$v$};
\draw (\vx,0) circle (\r);

\pgfmathsetmacro{\xint}{(\R*\R - \r*\r + \vx*\vx)/(2*\vx)}
\pgfmathsetmacro{\yint}{sqrt(\R*\R - \xint*\xint)}

\fill (\xint,\yint) circle (1.2pt);
\fill (\xint,-\yint) circle (1.2pt);

\node[right=3pt] at (\xint,-\yint) {$S_{\eps,\radius}$};
\node[right=3pt] at (\xint,\yint) {$R_{\eps,\radius}$};

\end{scope}

\end{tikzpicture}
\caption{The set $C_{\radius,\eps}$
defined in \eqref{eq:C_rho} and the points
$R_{\eps,\radius}$ (starting point) and $S_{\eps,\radius}$
(ending point)  defined in \eqref{eq:R_S}}
\label{fig:one}
\end{center}
\end{figure}

\medskip
	\textit{Step 1:}
To begin with, we first focus our attention to one interval $(a^i_{\eps,\radius},b^i_{\eps,\radius})=:(a_{\eps,\radius},b_{\eps,\radius})\subset(0,\ex)$. 
Since $\dot\Ga_\eps(x)\in C_{\rho,\eps}$ for all $x\in (a_{\eps,\rho},b_{\eps,\rho})$, 
the image of $\dot\Ga_\eps$ is contained in a simply connected arc of the circle. 
In particular, the corresponding angle (the lifting of $\dot\Ga_\eps$) satisfies
$$
|\theta_\eps(x)-\theta_\eps(y)|\le 2\pi
\qquad \text{for all } x,y\in (a_{\eps,\rho},b_{\eps,\rho}).
$$
Let $\tilde\theta_\eps$ be an arbitrary primitive of $\dot\theta_\eps$ on $(a_{\eps,\rho},b_{\eps,\rho})$. 
Define a new primitive $\prim$ by
$$
\prim(x):=\tilde\theta_\eps(x)-\min_{[a_{\eps,\rho},b_{\eps,\rho}]}\tilde\theta_\eps.
$$
Then $\prim\ge 0$ on $(a_{\eps,\rho},b_{\eps,\rho})$. Moreover, by the bound above, we have
$$
\sup_{(a_{\eps,\rho},b_{\eps,\rho})}\prim
=
\sup_{(a_{\eps,\rho},b_{\eps,\rho})}\tilde\theta_\eps
-
\min_{(a_{\eps,\rho},b_{\eps,\rho})}\tilde\theta_\eps
\le 2\pi,
$$
which yields
\begin{equation}\label{eq:angolo}
0\le \prim(x)\le 2\pi
\qquad \text{for all } x\in (a_{\eps,\rho},b_{\eps,\rho}).
\end{equation}

Let $\varphi\in C^{0,1}([0,\ex])$ be a test function.
Using the representation formula \eqref{eq:scal_k} and integrating by parts, we get
\begin{equation}\label{eq:duality_pairing}
\begin{aligned}
\int_{a_{\eps,\rho}}^{b_{\eps,\rho}} \kappa_\eps \, \varphi \, d\parameter
&=
\frac{\ex}{\lgae}\int_{a_{\eps,\radius}}^{b_{\eps,\radius}}\dot\theta_\eps(\parameter)\varphi(\parameter)\, d\parameter
\\
&
=
\frac{\pq}{\lgae}
\Big(
\varphi(b_{\eps,\radius})\prim(b_{\eps,\radius})-\varphi(a_{\eps,\radius})\prim(a_{\eps,\radius})
\Big)
-
\frac{\pq}{\lgae}
\int_{a_{\eps,\radius}}^{b_{\eps,\radius}} \prim(\parameter)\dot\varphi(\parameter)\, d\parameter 
\\
&=:\textrm{I}_{\eps,\radius}+\textrm{II}_{\eps,\radius}
\end{aligned}
\end{equation}
where, to simplify the notation, we write here and in the following $\kappa_\eps$ in place of $\kappa_{\Ga_\eps}$.
From \eqref{eq:angolo} we obtain 
\begin{equation}\label{eq:first}
\begin{aligned}
|\textrm{II}_{\eps,\radius}|
&\leq
\biggl|
\frac{\pq}{\lgae}
\int_{a_{\eps,\radius}}^{b_{\eps,\radius}} \prim(\parameter)\dot\varphi(\parameter)\, d\parameter
\biggr|
\leq
\frac{\pq}{\lgae}
\|\dot\varphi\|_{L^\infty}
\int_{a_{\eps,\radius}}^{b_{\eps,\radius}} |\prim(\parameter)|\, d\parameter
\\
&\leq 2\pi \frac{\pq}{\lgae}\|\dot\varphi\|_{L^\infty}(b_{\eps,\radius}-a_{\eps,\radius}).
\end{aligned}
\end{equation}
Now, we rewrite the boundary terms in \eqref{eq:duality_pairing} as
\begin{equation}\label{eq:boundary}
\textrm{I}_{\eps,\radius}
=
\frac{\pq}{\lgae}
\Big(
\big(\varphi(b_{\eps,\radius})-\varphi(a_{\eps,\radius})\big)\prim(b_{\eps,\radius})
+
\varphi(a_{\eps,\radius})\big(\prim(b_{\eps,\radius})-\prim(a_{\eps,\radius})\big)\Big)
=:\textrm{I}^{(1)}_{\eps,\radius}+\textrm{I}^{(2)}_{\eps,\radius}
\end{equation}
and using \eqref{eq:angolo} we estimate the first term as 
\begin{equation}
\begin{aligned}\label{eq:second}
|\textrm{I}_{\eps,\radius}^{(1)}|&=
\frac{\pq}{\lgae}
\big|
\big(\varphi(b_{\eps,\radius})-\varphi(a_{\eps,\radius})\big)\prim(b_{\eps,\radius})
\big|
\leq
2\pi\,\frac{\pq}{\lgae}|\varphi(b_{\eps,\radius})-\varphi(a_{\eps,\radius})|
\\
&
\leq
2\pi\,\frac{\pq}{\lgae}\|\dot\varphi\|_{L^\infty}(b_{\eps,\radius}-a_{\eps,\radius}).
\end{aligned}
\end{equation}
In order to estimate $\textrm{I}^{(2)}_{\eps,\radius}$
we distinguish between two geometric configurations: either $\dgae(a_{\eps,\radius})\neq \dgae(b_{\eps,\radius})$ or 
$\dgae(a_{\eps,\radius})=\dgae(b_{\eps,\radius})$.
If $\dgae(a_{\eps,\radius})=\dgae(b_{\eps,\radius})$ 
then 
$\theta_\eps(b_{\eps,\radius})-\theta_\eps(a_{\eps,\radius})=0$,
and $\textrm{I}^{(2)}_{\eps,\radius}=0$. 
Combining \eqref{eq:duality_pairing}, \eqref{eq:boundary}, \eqref{eq:first}, and \eqref{eq:second}, we obtain
\begin{equation}\label{pi}
\biggl|
\int_{a_{\eps,\radius}}^{b_{\eps,\radius}} \kappa_\eps \, \varphi \, d\parameter
\biggr|
\leq |\textrm{I}_{\eps,\radius}^{(1)}|+|\textrm{II}_{\eps,\radius}|
\leq C\|\dot\varphi\|_{L^\infty}(b_{\eps,\radius}-a_{\eps,\radius}).
\end{equation}

If instead $$\dgae(a_{\eps,\radius})\neq \dgae(b_{\eps,\radius})$$ we will have $$\dgae(a_{\eps,\radius})=R_{\eps,\radius} \text{ and }\dgae(b_{\eps,\radius})=S_{\eps,\radius} \text{ or viceversa. }$$
In the former case 
\begin{equation*}
\theta_\eps(b_{\eps,\radius})-\theta_\eps(a_{\eps,\radius}) =  2\pi - d_{\eps,\radius},
\end{equation*}
where $d_{\eps,\radius}\geq0$ is a function depending on $\radius$ and $\eps$ and such that 
$d_{\eps,\radius}\leq |o_{\radius}(1)|$ with $o_{\radius}(1)$ being independent of $\eps$ and vanishing as $\radius\rightarrow 0^+$. Instead, if $\dgae(a_{\eps,\radius})=S_{\eps,\radius}$ and $\dgae(b_{\eps,\radius})=R_{\eps,\radius}$ it will be 
\begin{equation}\label{eq:salto}
	\theta_\eps(b_{\eps,\radius})-\theta_\eps(a_{\eps,\radius}) =  -2\pi  +d_{\eps,\radius}.
\end{equation}
Therefore, in the first case
\begin{equation*}
\varphi(a_{\eps,\radius})\big(\prim(b_{\eps,\radius})-\prim(a_{\eps,\radius})\big)
=
 2\pi\,\varphi(a_{\eps,\radius}) -d_{\eps,\radius}\varphi(a_{\eps,\radius})
=
 2\pi\,\varphi({a_{\eps,\radius}}) - d_{\eps,\radius}\varphi(a_{\eps,\radius})
\end{equation*}
while, in the second one
\begin{equation*}
\varphi(a_{\eps,\radius})\big(\prim(b_{\eps,\radius})-\prim(a_{\eps,\radius})\big)
=
 -2\pi\,\varphi(a_{\eps,\radius}) +d_{\eps,\radius}\varphi(a_{\eps,\radius})
=
 -2\pi\,\varphi({a_{\eps,\radius}}) + d_{\eps,\radius}\varphi(a_{\eps,\radius}).
\end{equation*} 
Hence
\begin{equation*}
    \textrm{I}_{\eps,\radius}^{(2)}
    =
    \frac{\pq}{\lgae} (\pm 2\pi\,\varphi(a_{\eps,\radius}) \mp d_{\eps,\radius}\varphi(a_{\eps,\radius})).
\end{equation*}
We conclude
\begin{equation}\label{pi+}
	\biggl|
	\int_{a_{\eps,\radius}}^{b_{\eps,\radius}} (\kappa_\eps\mp2\pi\delta_{a_{\eps,\radius}})\, \varphi \, d\parameter
	\biggr|
	\leq C\|\dot\varphi\|_{L^\infty}(b_{\eps,\radius}-a_{\eps,\radius})+Cd_{\eps,\radius}\|\varphi\|_{L^\infty}.
\end{equation}

\textit{Step 2:} 
We want now to estimate how many intervals $(a_{\eps,\radius}^i,b_{\eps,\radius}^i)$ as in the previous step satisfy  $\dgae(a^i_{\eps,\radius})\neq \dgae(b^i_{\eps,\radius})$; for each of such intervals we have, using \eqref{eq:salto}, 
\begin{align*}
	2\pi-d_{\eps,\radius}= 
\left|\int_{a^i_{\eps,\radius}}^{b^i_{\eps,\radius}}\kappa_\eps\,d\parameter\right|\leq \int_{a^i_{\eps,\radius}}^{b^i_{\eps,\radius}}|\kappa_\eps| \,d\parameter
\end{align*}
and therefore, if we denote by $N_{\eps,\radius}$ the number of such intervals, since we can assume $d_{\eps,\radius}<\pi$ we conclude, using \eqref{eq:kx_uniformly}, that 
\begin{equation}
N_{\eps,\radius}\leq \frac1\pi\int_{E^\eps_\radius}|\kappa_\eps|\, d\parameter\leq \frac {C}{\pi\radius}.
\end{equation}
In particular, for $\radius=\frac18$ we find $N_{\eps,\frac18}\leq \widehat C$ for an absolute constant $\widehat C>0$.

We claim that 
$$ \text{ for all }
\radius',\radius''\in\Big(0,\frac18\Big) \text{ with }\radius'<\radius'', \text{ it holds } N_{\eps,\radius'}\leq N_{\eps,\radius''},$$
and so
\begin{equation}\label{number_deltas}
	N_{\eps,\radius}\leq \widehat C \qquad\forall \radius\in \Big(0,\frac18\Big),
\end{equation}
for all $\eps$ small enough. To see that $N_{\eps,\radius'}\leq N_{\eps,\radius''}$ it is enough  to observe that if $\dgae(a_{\eps,\radius'})=R_{\eps,\radius'}$ and $\dgae(b_{\eps,\radius'})=S_{\eps,\radius'}$ (we argue similarly if the opposite situation holds),  then by continuity of $\dgae$ there is a subinterval $(a_{\eps,\radius''},b_{\eps,\radius''})\subset (a_{\eps,\radius'},b_{\eps,\radius'})$ with $\dgae(a_{\eps,\radius''})=R_{\eps,\radius''}$ and $\dgae(b_{\eps,\radius''})=S_{\eps,\radius''}$. This concludes Step 2.

As a consequence of the previous discussion, we can select a not-relabelled infinitesimal subsequence of $\eps$ such that $$N_{\eps,\radius}=N \text{ is constant and does not depend on } \eps \text{ and }\radius.$$ 

Let now $(a_{\eps,\radius}^i,b_{\eps,\radius}^i)$ for $i=1,\dots,N$ be the aforementioned intervals and let $(a_{\eps,\radius}^i,b_{\eps,\radius}^i)$ for $i>N$ denote the other intervals satisfying $\dgae(a_{\eps,\radius})=\dgae(b_{\eps,\radius})$ so that \eqref{pi} holds; we set
$$\omega_{\eps,\radius}:= 2\pi\sum_{i=1}^{N}\alpha_i\delta_{a^i_{\eps,\radius}}$$ 
where $\alpha_i$ is $\pm1$ according to the case that $\dgae(a_{\eps,\radius}^i)=R_{\eps,\radius}$ and $\dgae(b_{\eps,\radius}^i)=S_{\eps,\radius}$ or viceversa, respectively
\footnote{Choosing $b_{\eps,\rho}^i$ instead of $a_{\eps,\rho}^i$ in the definition of $\omega_\eps$ would lead to the same limit in the flat norm. Indeed, for every test function $\varphi\in C^{0,1}([0,\pq])$ with $\|\varphi\|_{C^{0,1}}\le 1$, one has
\[
\bigl|\delta_{a_{\eps,\rho}^i}(\varphi)-\delta_{b_{\eps,\rho}^i}(\varphi)\bigr|
=
|\varphi(a_{\eps,\rho}^i)-\varphi(b_{\eps,\rho}^i)|
\le |b_{\eps,\rho}^i-a_{\eps,\rho}^i|,
\]
and therefore
\[
\bigl\|\delta_{a_{\eps,\rho}^i}-\delta_{b_{\eps,\rho}^i}\bigr\|_{\mathrm{flat}}
\le |b_{\eps,\rho}^i-a_{\eps,\rho}^i|.
\]
Since the total length of the intervals $(a_{\eps,\rho}^i,b_{\eps,\rho}^i)$ tends to zero as $\eps\to0^+$ for fixed $\rho$, the choice of the left or right endpoint is equivalent in the limit.}
.
Notice carefully that $\omega_{\eps,\radius}\in M_{{\rm fin}, \mathbb Z}
([0,\ex])$ and that 
\begin{equation*}
    |\omega_{\eps,\radius}|([0,\ex])\leq 2\pi N\leq C
\end{equation*}
where $C$ is independent of $\eps,\radius$.
Using \eqref{pi+}, we finally estimate
\begin{align}\label{estimate_1}
	\biggl|
	\int_{\cup_{i=1}^N(a_{\eps,\radius}^i,b_{\eps,\radius}^i)} (\kappa_\eps -\omega_{\eps,\radius})\, \varphi \, d\parameter
	\biggr|
	\leq C|E^\eps_\radius|+CN d_{\eps,\radius}\leq C\eps^{\frac12}+Cd_{\eps,\radius}\leq C\eps^{\frac12}+C|o_{\radius}(1)|, 
\end{align}
for all $\varphi\in C^{0,1}([0,\pq])$ with $\|\varphi\|_{W^{1,\infty}}\leq 1$, where we have used also \eqref{eq:grandezza_E} and \eqref{number_deltas}. 

%

\textit{Step 3:} We now estimate the flat norm of $\kappa_\eps$ on $[0,\pq]\setminus \cup_{i=1}^N(a_{\eps,\radius}^i,b_{\eps,\radius}^i)$.
We shall show that $\exists \, C> 0$ such that for any $\radius\in(0,\frac 18)$ and any $\eps\in (0,1]$ 
\begin{equation}\label{eq:complementare}
	\biggl|
	\int_{[0,\pq]\setminus \cup_{i=1}^N(a_{\eps,\radius}^i,b_{\eps,\radius}^i)}  \kappa_\eps \varphi \, d\parameter
	\biggr|
	\leq
    C\radius
\end{equation}
for any $\varphi\in C^{0,1}([0,\ex])$, $\|\varphi\|_{W^{1,\infty}}\leq 1$.
Up to relabelling, assume that $0< a_{\eps,\radius}^1< b_{\eps,\radius}^1\leq a_{\eps,\radius}^2<b_{\eps,\radius}^2\leq\dots\leq a_{\eps,\radius}^N<b_{\eps,\radius}^N< \pq$, and so we might write
$$[0,\pq]\setminus \cup_{i=1}^N(a_{\eps,\radius}^i,b_{\eps,\radius}^i)=\bigcup_{i=1}^{N+1}[b^{i-1}_{\eps,\radius},a^i_{\eps,\radius}]$$ where we have set
$b^0_{\eps,\radius}=0$ and $a^{N+1}_{\eps,\radius}=\pq$.
%
Arguing as in \eqref{eq:duality_pairing},\eqref{eq:boundary}, denoting by $\prim$ a primitive of $\dot\theta_\eps$ on $({b^{i-1}_{\eps,\radius}},{a^i_{\eps,\radius}})$ we have
\begin{equation}
\begin{aligned}\label{eq:duality_pairing_bis}
\int_{b^{i-1}_{\eps,\radius}}^{a^i_{\eps,\radius}} \kappa_\eps \, \varphi \, d\parameter
&=
\frac{\pq}{\lgae} \Big(
\varphi(a^i_{\eps,\radius})\,\prim(a^i_{\eps,\radius}) - \varphi(b^{i-1}_{\eps,\radius})\,\prim(b^{i-1}_{\eps,\radius})
\Big)
-
\frac{\pq}{\lgae} \int_{b^{i-1}_{\eps,\radius}}^{a^i_{\eps,\radius}} \prim(\parameter)\dot\varphi(\parameter)\, d\parameter
\\
&=
\frac{\pq}{\lgae} \Big[
\big(\varphi(a^i_{\eps,\radius})-\varphi(b^{i-1}_{\eps,\radius})\big)\prim(a^i_\eps)
+
\varphi(b^{i-1}_\eps)\big(\prim(a^i_{\eps,\radius})-\prim(b^{i-1}_{\eps,\radius})\big)
\Big]
\\
&\quad
-
\frac{\pq}{\lgae} \int_{b^{i-1}_{\eps,\radius}}^{a^i_{\eps,\radius}}\prim(\parameter)\dot\varphi(\parameter)\, d\parameter .
\end{aligned}
\end{equation}
By definition of the intervals $(a^i_{\eps,\radius},b^i_{\eps,\radius})$,  
\begin{equation*}
|\theta_\eps(a^i_{\eps,\radius})-\theta_\eps(b^{i-1}_{\eps,\radius})| \le C\radius.
\end{equation*}
Consequently, the second boundary term in \eqref{eq:duality_pairing_bis} can be estimated as
\begin{equation}\label{eq:uno}
\frac{\pq}{\lgae}
\big|\varphi(b^{i-1}_{\eps,\radius})\big(\prim(a^i_{\eps,\radius})-\prim(b^{i-1}_{\eps,\radius})\big)\big|
\le C\frac{\pq}{\lgae}\|\varphi\|_{L^\infty}\,\radius.
\end{equation}
Moreover,
\begin{equation}\label{eq:due}
\begin{aligned}
\frac{\pq}{\lgae}
\big|
\big(\varphi(a^i_{\eps,\radius})-\varphi(b^{i-1}_{\eps,\radius})\big)\prim(a^i_{\eps,\radius})
\big|
&\leq
\frac{\pq}{\lgae}
\|\dot\varphi\|_{L^\infty}(a^i_{\eps,\radius}-b^{i-1}_{\eps,\radius})\,|\prim(a^{i-1}_{\eps,\radius})|
\\
&
\leq
C
\frac{\pq}{\lgae}
\radius\,\|\dot\varphi\|_{L^\infty}(a^i_{\eps,\radius}-b^{i-1}_{\eps,\radius}).
\end{aligned}
\end{equation}
Eventually,
\begin{align}\label{eq:tre}
\biggl|
\int_{b^{i-1}_{\eps,\radius}}^{a^i_{\eps,\radius}} \prim(\parameter)\dot\varphi(\parameter)\, d\parameter
\biggr|
=\biggl|
\int_{(b^{i-1}_{\eps,\radius},a^i_{\eps,\radius})\setminus E^\eps_\radius} \prim(\parameter)\dot\varphi(\parameter)\, d\parameter
\biggr|
\leq
C\radius\,\|\dot\varphi\|_{L^\infty}(a^i_{\eps,\radius}-b^{i-1}_{\eps,\radius}).
\end{align}

\nada{
Eventually,
\begin{align}\label{eq:tre}
\biggl|
\int_{b^{i-1}_{\eps,\radius}}^{a^i_{\eps,\radius}} \dot\varphi(\parameter)\prim(\parameter)\, d\parameter
\biggr|&\leq\biggl|
\int_{(b^{i-1}_{\eps,\radius},a^i_{\eps,\radius})\setminus E^\eps_\radius} \dot\varphi(\parameter)\prim(\parameter)\, d\parameter
\biggr|+\biggl|
\int_{(b^{i-1}_{\eps,\radius},a^i_{\eps,\radius})\cap E^\eps_\radius} \dot\varphi(\parameter)\prim(\parameter)\, d\parameter
\biggr|\nonumber\\
&
\leq
\radius\,\|\dot\varphi\|_{L^\infty}(a^i_{\eps,\radius}-b^{i-1}_{\eps,\radius})+2\pi\|\dot\varphi\|_{L^\infty}|(b^{i-1}_{\eps,\radius},a^i_{\eps,\radius})\cap E_\radius^\eps|,
\end{align}
where to estimate the second term in the second inequality we have used that $\prim\leq 2\pi$... {\color{red} OCCHIO a questo $2\pi$}
}

Collecting \eqref{eq:duality_pairing_bis}, \eqref{eq:uno},\eqref{eq:due} and \eqref{eq:tre} we obtain, for a absolute constant $C>0$,
\begin{equation*}
\biggl|
\int_{b^{i-1}_{\eps,\radius}}^{a^i_{\eps,\radius}}  \kappa_\eps \varphi \, d\parameter
\biggr|
\leq
C\radius+
C\,\radius\,(a^i_{\eps,\radius}-b^{i-1}_{\eps,\radius})
\end{equation*}
and so, summing over $i=1,\dots,N+1$, we conclude
\begin{equation}\label{eq:complementare_bis}
	\biggl|
	\int_{[0,\pq]\setminus \cup_{i=1}^N(a_{\eps,\radius}^i,b_{\eps,\radius}^i)}  \kappa_\eps \varphi \, d\parameter
	\biggr|
	\leq
    C\radius (N+1)+
	C\,\radius
    \leq C\radius
\end{equation}
where we recall that $N$ is constant and does not depend on $\eps$ and $\radius$; this proves \eqref{eq:complementare}.
Using \eqref{estimate_1} and \eqref{eq:complementare_bis}, we arrive at
\begin{equation*}
	\biggl|
	\int_{0}^\pq  (\kappa_\eps -\omega_{\eps,\radius})\varphi \, d\parameter
	\biggr|
	\leq C\radius
	+C\eps^{\frac12}+C|o_{\radius}(1)|\,,
\end{equation*}
for all $\varphi\in C^{0,1}([0,\pq])$ with $\|\varphi\|_{W^{1,\infty}}\leq 1$,
which shows that 
 for every fixed $\radius\in (0,\frac18)$,
\begin{equation*}
\limsup_{\eps\to 0^+}
\bigl\| \kappa_\eps - \omega_{\eps,\radius} \bigr\|_{\mathrm{flat},[0,\ex]}
\le C(\radius+o_{\radius}(1)).
\end{equation*}
Passing to further subsequence, we can assume that 
\begin{equation}\label{eq:ai_eps_to_ai}
   a_{\eps,\radius}^i\rightarrow a_\radius^i\in [0,\pq] \quad \text{for all} \quad i=1,\dots,N.
\end{equation}
 Therefore we easily conclude, as $\eps\rightarrow 0^+$,
\begin{align}\label{muconv}
\omega_{\eps,\radius}\to \omega_\radius:=2\pi\sum_{i=1}^{N} \alpha_i\delta_{a_\radius^i}
\qquad \qquad\text{ in the flat norm,  }
\end{align}
with
\begin{equation}\label{eq:var}
    |\omega_{\radius}|([0,\ex])\leq 2\pi N\leq C.
\end{equation}
Now up to a subsequence, by \eqref{eq:var}, there exists $\omega\in M_{{\rm fin}, \mathbb Z}
([0,\ex])$ such that 
$$
\omega_\radius\to\omega \quad\text{ in the flat norm}.$$
Combining the previous inequality  with \eqref{muconv}
we get
\begin{align*}
	\limsup_{\eps\to 0^+}
	\bigl\| \kappa_\eps - \omega \bigr\|_{\mathrm{flat},[0,\ex]}
    &\leq
    \limsup_{\eps\to 0^+} \big(
	\bigl\| \kappa_\eps - \omega_{\eps,\radius} \bigr\|_{\mathrm{flat},[0,\ex]}+
	\bigl\| \omega_{\eps,\radius}-\omega_{\radius} \bigr\|_{\mathrm{flat},[0,\ex]}
    +
	\bigl\| \omega_{\radius}-\omega \bigr\|_{\mathrm{flat},[0,\ex]}\big)
    \\
    &\leq
    C(\radius+o_{\radius}(1))+
	\bigl\| \omega_{\radius}-\omega \bigr\|_{\mathrm{flat},[0,\ex]}.
\end{align*}
Since $\radius\in (0,\frac18)$ is arbitrary, we then obtain that
\begin{equation*}
\lim_{\eps\to 0^+}
\bigl\| \kappa_\eps - \omega \bigr\|_{\mathrm{flat},[0,\ex]} = 0.
\end{equation*}
\end{proof}

Let us define a strictly increasing odd
function $\Phi\in C^1(\mathbb{R})$ by 
\begin{equation}\label{eq:Phi}
    \Phi(\theta):=\int_0^\theta 2\sqrt{1-\cos\phi}\,d\phi.
\end{equation}
\begin{Lemma}\label{lem:theta_w}
Let $\theta_\eps$ be as is \eqref{eq:scal_k}. We have 
\begin{equation*}
\lim_{\eps\to 0^+}
    \Big\|\frac{2\pi}{8\sqrt{2}} \partial_x (\Phi\circ\theta_\eps)-\partial_x\theta_\eps\Big\|_{\mathrm{flat},[0,\ex]}
    =0.
\end{equation*}
In particular, letting $\omega\in M_{{\rm fin},\mathbb Z}([0,\ex])$
be the measure for which \eqref{eq:compactness_in_flat_norm} holds, we have
\begin{equation}\label{eq:theta_mu}
\frac{2\pi}{8\sqrt{2}}\partial_x(\Phi\circ\theta_{\eps_k})\to\omega
\qquad \text{in the flat norm}.
\end{equation}
\end{Lemma}
\begin{proof}
Set
\begin{equation*}
g(\theta):=\frac{2\pi}{8\sqrt{2}}\Phi(\theta)-\theta.
\end{equation*}
Since
$\Phi(\theta+2\pi)
=\Phi(\theta)+\displaystyle\int_0^{2\pi}2\sqrt{1-\cos\phi} \,d\phi
=\Phi(\theta)+8\sqrt{2},$
it follows that
$g(\theta+2\pi)=g(\theta)$ for all $\theta\in\mathbb R$.
Hence $g$ is continuous and $2\pi$-periodic and, 
therefore, it is bounded on $\mathbb R$. Let
\begin{equation*}
M:=\|g\|_{L^\infty(\mathbb R)}<+\infty.
\end{equation*}
Moreover, for every $k\in\mathbb Z$,
\begin{equation*}
g(2\pi k)
=
\frac{2\pi}{8\sqrt{2}}\Phi(2\pi k)-2\pi k
=0.
\end{equation*}
As a consequence,
\begin{equation}\label{eq:lim_sup_g}
\lim_{\delta\to 0^+}\sup\bigl\{|g(\theta)|:\text{dist}(\theta,2\pi\mathbb Z)\le \delta\bigr\} =0.
\end{equation}
For every
$\varphi\in C^{0,1}([0,\pq])$ with $\|\varphi\|_{W^{1,\infty}}\le 1$, integrating by parts we obtain
\begin{align*}
\left|\left\langle
\frac{2\pi}{8\sqrt{2}} \partial_x (\Phi\circ\theta_\eps)-\partial_x\theta_\eps,
\varphi
\right\rangle \right|
=
\left|
-\int_0^\pq g(\theta_\eps(x))\,\dot\varphi(x)\,dx
\right|
\le
\|g(\theta_\eps)\|_{L^1([0,\pq])}
.
\end{align*}
Taking the supremum over all $\varphi$ yields
\begin{equation}\label{eq:minore_g_L1}
\Big\|
\frac{2\pi}{8\sqrt{2}} \partial_x (\Phi\circ\theta_\eps)-\partial_x\theta_\eps
\Big\|_{\mathrm{flat},[0,\ex]}
\le
\|g(\theta_\eps)\|_{L^1([0,\pq])}.
\end{equation}
Using \eqref{eq:grandezza_E},
\begin{equation}\label{eq:g-on-E}
\int_{E_\radius^\eps} |g(\theta_\eps)|\,dx
\le
M\,|E_\radius^\eps|
\le
\frac{CM}{\radius^2}\eps^{1/2}.
\end{equation}

On the other hand, if $x\notin E_\radius^\eps$, then
$|\dgae(x)-v|\le \radius$,
where we recall that $v=\frac{p-q}{\ex}$.
Since
$\dgae(x)
=
\frac{\lgae}{\pq}(\cos\theta_\eps(x),\sin\theta_\eps(x)),$
and $\lgae/\pq\to 1$, it follows that for $\eps>0$ small enough the angle $\theta_\eps(x)$ must be close to some multiple of $2\pi$, uniformly in $x\in [0,\pq]\setminus E_\radius^\eps$. More precisely, there exists a quantity $\delta_\radius^\eps\ge 0$ such that
\begin{equation*}\label{eq:theta-close}
\textrm{dist}(\theta_\eps(x),2\pi\mathbb Z)\le \delta_\radius^\eps
\qquad \forall x\in [0,\pq]\setminus E_\radius^\eps,
\end{equation*}
and
\begin{equation}\label{eq:delta-rho}
\limsup_{\eps\to 0^+}\delta_\radius^\eps\le C\radius.
\end{equation}
Hence, using \eqref{eq:g-on-E}, we get
\begin{equation*}
\|g(\theta_\eps)\|_{L^1([0,\pq])}
\le
\frac{CM}{\radius^2}\eps^{1/2}
+
\pq \sup\bigl\{|g(\theta)|:\text{dist}(\theta,2\pi\mathbb Z)\le \delta^\eps_\radius\bigr\},
\end{equation*}
from which, recalling \eqref{eq:delta-rho}
\begin{equation*}
\limsup_{\eps\to 0^+}\|g(\theta_\eps)\|_{L^1([0,\pq])}
\le
\pq \sup\bigl\{|g(\theta)|:\text{dist}(\theta,2\pi\mathbb Z)\le C\radius\}.
\end{equation*}
Letting $\radius\to 0^+$ and recalling \eqref{eq:lim_sup_g}, we conclude that
\begin{equation*}
\lim_{\eps\to 0^+}\|g(\theta_\eps)\|_{L^1([0,\pq])}=0.
\end{equation*}
This, together with \eqref{eq:minore_g_L1} gives the first part of the statement.
Now, Proposition~\ref{lem:compactness_in_flat_norm} implies
$$\partial_x\theta_{\eps_k}\to\omega
\qquad\text{in the flat norm},$$
hence \eqref{eq:theta_mu} follows as well.
\end{proof}
\subsection{\texorpdfstring{$\Gamma$-liminf inequality}{Gamma-liminf inequality}}\label{sec:lb}
In this section we prove the $\Gamma-\liminf$ inequality of Theorem \ref{thm:main_result}. To this purpose we may take, without loss of generality, a sequence $(\gae,\kae)$ such that 
\begin{equation*}
   \lgae-|p-q|\leq C{\eps^{1/2}},\qquad \qquad \int_{\gamma_\eps}\eps^{1/2}\kae^2 \ d s\leq C,
\end{equation*}
which, from Lemma \ref{lem:H1} and Proposition \ref{lem:compactness_in_flat_norm}, imply $\Ga_\eps\to\Ga$ strongly in $H^1([0,\pq];\mathbb{R}^2)$ and $\kappa_{\Ga_\eps}$ converge in the flat norm to $\omega=2\pi\sum_{j=1}^{N} c_j\delta_{x^j}$, with $N\geq 0$, 
$0\leq x_0<x_1< \dots< x_N\leq\pq$ and $c_j\in \mathbb Z$. We have to prove that 
\begin{equation*}\label{eq:liminf}
    \liminf_{\eps\to0^+}
G_\eps(\Ga_\eps)
\ge
\sigma\, G(\Ga,\omega).
\end{equation*}
The proof relies on Remark \ref{rem:phase_transition}.
We can use
the following \textit{Modica-Mortola trick} to estimate
\begin{align*}
G_\eps(\Ga_\eps)
&= \frac{\pq}{\lgae}
\int_0^{\pq} \eps^{1/2}|\partial_\parameter\theta_\eps|^2\,d\parameter
+
\frac{\lgae}{\pq}
\int_0^{\pq} \frac{1}{\eps^{1/2}} (1-\cos\theta_\eps)\,d\parameter
\geq
\int_0^{\pq} 2|\partial_\parameter\theta_\eps| \sqrt{1-\cos\theta_\eps} \,d\parameter 
,
\end{align*}
where we have simply used the algebraic inequality $\eps^{1/2}a^2+\frac{1}{\eps^{1/2}} b^2\geq 2|a||b|$, with $a=\sqrt{\frac{\pq}{\ell(\gae)}}|\partial_\parameter\theta_\eps|$ and $b=\sqrt{\frac{\ell(\gae)}{\pq}}\sqrt{1-\cos\theta_\eps}$.
By the definition of the function $\Phi$ in \eqref{eq:Phi}, its derivative is
\begin{equation*}
\Phi'(\theta)=2\sqrt{1-\cos\theta},
\end{equation*}
and therefore
\begin{equation*}\label{eq:G_var}
G_\eps(\Ga_\eps)
\geq
\int_0^{\pq} 2|\partial_\parameter\theta_\eps| \sqrt{1-\cos\theta_\eps} \,d\parameter
= 
\int_0^{\pq} |\partial_\parameter(\Phi\circ\theta_\eps)| \,d\parameter = |\partial_\parameter(\Phi\circ\theta_\eps)|([0,\pq]).
\end{equation*}
Recalling from \eqref{eq:condizioni_bordo} that $\theta_\eps(0)=0$,
it follows that the sequence $(\Phi\circ\theta_\eps)_\eps$ is uniformly bounded in
$BV([0,\pq]))$. Hence, up to a further subsequence, there exists a function
$u\in BV([0,\pq])$ such that
\begin{equation*}
\Phi\circ\theta_\eps \to u
\qquad \text{in }L^1([0,\pq]).
\end{equation*}
Hence the sequence $\big(\partial_x(\Phi\circ\theta_\eps)\big)$ converges weakly* as measures to $\partial_x u$ and,
by the lower semicontinuity of the total variation with respect to $L^1$ convergence, we deduce
\begin{equation}\label{eq:lsc_BV}
\liminf_{\eps\to0^+} G_\eps(\Ga_\eps)
\ge
\liminf_{\eps\to0^+} |\partial_x(\Phi\circ\theta_\eps)|([0,\ex])
\ge
|\partial_x u|([0,\pq]).
\end{equation}
It remains to identify the structure of the limit function $u$.  
Since \eqref{eq:theta_mu} holds, we deduce that 
\begin{equation*}
\partial_x u
=
\frac{\sigma}{2\pi}\omega
=
\sigma\sum_{j=1}^N c_j\,\delta_{x_j}.
\end{equation*}
Therefore,
\begin{equation}\label{eq:Du_total_variation}
|\partial_\parameter u|([0,\pq])
=
\sigma\sum_{j=1}^N |c_j|.
\end{equation}
From \eqref{eq:lsc_BV} and \eqref{eq:Du_total_variation} we get 
\begin{equation*}\label{eq:sci_quasi}
\liminf_{\eps\to0} G_{\eps}(\Ga_{\eps})
\ge
\sigma\sum_{j=1}^N |c_j|
=\sigma\,G(\Ga,\omega).
\end{equation*}

\subsection{\texorpdfstring{$\Gamma$-limsup inequality}{Gamma-limsup inequality}}\label{sec:up}
We begin by recalling some basic facts concerning the (planar) borderline elastica (see also \cite{Mi:20,MiRu:25}), which will be used in the construction of the recovery sequence.
\paragraph{Borderline elastica.}\label{par:borderline_elastica}
A prototypical arclength parametrization is given by
\begin{equation*}
\alpha_\eps(s)=(\alpha_{1,\eps}(s),\alpha_{2,\eps}(s)), 
\qquad s\in\mathbb{R},
\end{equation*}
with
\begin{equation*}
\alpha_{1,\eps}(s)
=
s-2\sqrt{2\eps}\,
\tanh\!\left(\frac{s}{\sqrt{2\eps}}\right),
\qquad
\alpha_{2,\eps}(s)
=
2\sqrt{2\eps}\,
\sech\!\left(\frac{s}{\sqrt{2\eps}}\right).
\end{equation*}
The lifting of $\alpha_\eps'$ is given by 
\begin{equation}\label{eq:angolo_tangente}
   \theta_\eps(s)=4\arctan (e^{s/\sqrt{2\eps}})
\end{equation}
and the curvature is
\begin{equation}\label{eq:curvatura}
\kappa_{\alpha_\eps}(s)
=
\partial_s\theta_\eps(s)
=
\frac{\sqrt{2}}{\sqrt{\eps}}\,
\sech\!\left(\frac{s}{\sqrt{2\eps}}\right)
=
\frac{4}{\sqrt{2\eps}}\frac{e^{s/\sqrt{2\eps}}}{1+e^{2s/\sqrt{2\eps}}}
.
\end{equation}
Note that $\kappa_{\alpha_\eps}> 0$. This satisfies $-\,\kappa_{\alpha_\eps} + \eps(2\partial^2_s\kappa_{\alpha_\eps} + \kappa_{\alpha_\eps}^3)=0$.

\paragraph{Key construction.}\label{par:key_construction}
\nada{Recall that $\pq:=|p-q|$ and $s$ denotes the arc-length parameter on $[0,\lgae]$ of $\gamma_\eps$. When we rescale $\gamma_\eps$ on $[0,\pq]$, without relabelling, so that $|\dot\gamma_\eps|=\frac{\lgae}{\pq}\geq 1$, we call $\parameter:=s \frac{\pq}{\lgae}$ the new variable.} 

Here we define the building block which will be used to obtain the recovery sequence. 

Let $\eps<<1$ and $\de$ be such that
\begin{equation}\label{eq:scelta_de}
   \de=\eps^{a} \text{ with } a\in(\tfrac14,\tfrac12). 
\end{equation}
 Consider the elastica 
\begin{equation*}
    \alpha_\eps(s)=(s-2\sqrt{2\eps}\,
\tanh\!\left(\frac{s}{\sqrt{2\eps}}\right),2\sqrt{2\eps}\,
\sech\!\left(\frac{s}{\sqrt{2\eps}}\right)), \qquad s\in[-\de,\de],
\end{equation*}
and denote with $\theta_\eps$ the lifting defined in \eqref{eq:angolo_tangente}. We aim to extend the elastica $\alpha_\eps$ beyond the interval $[-\de,\de]$ by attaching appropriate circular arcs at both endpoints.  
More precisely, on the left (respectively, right) side we prolong the curve by a circular arc $\raccordo$ with constant curvature equal to $\kappa_{\alpha_\eps}(-\de)$ (respectively, $\kappa_{\alpha_\eps}(\de)$).  
Each arc is continued up to the points $p_\eps$ and $q_\eps$, respectively, where the tangent vector becomes horizontal. Let us call $\keycon$ the resulting curve parametrized by arclength.

We focus on the construction of the circular arc attached to the left endpoint of the elastica.  
The right arc is obtained in a completely analogous way by symmetry and yields the same energetic contribution.
We begin by observing that the radius of the circle containing the arc is given by
\begin{equation}\label{eq:raggio}
r_\eps
=
\frac{1}{\kappa_{\alpha_\eps}(-\de)}
=
\frac{\sqrt{2\eps}}{4}
\frac{1+e^{-2\de/\sqrt{2\eps}}}{e^{-\de/\sqrt{2\eps}}},
\end{equation}
where we have used \eqref{eq:curvatura}.
Next, we consider the central angle of the arc.   
At both points $p_\eps$ and $\alpha_\eps(-\de)$, the radius of the circle is perpendicular to the tangent to the curve. In particular, at $p_\eps$ the tangent is horizontal, so the corresponding radius is vertical. At $\alpha_\eps(-\de)$, the tangent forms an angle $\theta_\eps(-\de)$ with the horizontal, and therefore the radius at this point forms the same angle $\theta_\eps(-\de)$ with the vertical.
Consequently, the central angle of the circular arc coincides with $\theta_\eps(-\de)$.
It then follows that the length of the circular arc is
\begin{equation}\label{eq:lunghezza_raccordo}
\ell(\raccordo) 
= r_\eps \, \theta_\eps(-\de)
= \sqrt{2\eps}
\frac{1+e^{-2\de/\sqrt{2\eps}}}{e^{-\de/\sqrt{2\eps}}}
\arctan (e^{-\de/\sqrt{2\eps}}),
\end{equation}
where we substituted the expressions for $r_\eps$ and $\theta_\eps(-\de)$ from \eqref{eq:raggio} and \eqref{eq:angolo_tangente}.

Moreover, by applying the chord theorem, we deduce that
\begin{equation*}
p_\eps = \bigl(\alpha_{1,\eps}(-\de) - r_\eps \, \sin\theta_\eps(-\de), \, y_{p_\eps}\bigr),
\qquad
q_\eps = \bigl(\alpha_{1,\eps}(\de) + r_\eps \, \sin\theta_\eps(\de), \, y_{q_\eps}\bigr),
\end{equation*}
for a suitable $y_{p_\eps}$ that we do not need to identify, 
and hence, 
by the symmetry of the construction,
we get
\begin{equation}\label{eq:peps-qeps}
    |p_\eps-q_\eps|=2\alpha_{1,\eps}(\de)+2r_\eps \, \sin\theta_\eps(\de)
    = 2\de-4\sqrt{2\eps}\,
\tanh\!\left(\frac{\de}{\sqrt{2\eps}}\right)+2r_\eps \, \sin\theta_\eps(\de).
\end{equation}

Now, recalling \eqref{eq:F_eps}, for each connecting arc $\raccordo$ we have 
\begin{equation}\label{eq:energia_arco}
    F_\eps(\raccordo)=\ell(\raccordo)+\eps\int_{\raccordo} \kappa_{\raccordo}^2 \,ds=\ell(\raccordo)+\eps\frac{\theta_\eps(-\de)}{r_\eps}.
\end{equation}
Clearly 
\begin{equation}\label{eq:lunghezza_ricciolo}
\ell(\alpha_\eps,[-\de,\de]) = \int_{-\de}^{\de} \sqrt{ (\alpha_{1,\eps}'(s))^2 + (\alpha_{2,\eps}'(s))^2 } \, ds= 2 \de.
\end{equation}
Furthermore, firstly using \eqref{eq:curvatura} and then performing the change of variable $u_\eps = s / \sqrt{2\eps}$ , we have 
\begin{equation}
\begin{aligned}\label{eq:eps_k_squared_ricciolo}
\eps \int_{-\de}^{\de} \kappa_{\alpha_\eps}^2 \, ds
&= 2 \int_{-\de}^{\de} \sech^2\Big(\frac{s}{\sqrt{2\eps}}\Big) ds
=2 \sqrt{2\eps} \int_{-\de/\sqrt{2\eps}}^{\de/\sqrt{2\eps}} \sech^2(u_\eps) \, du_\eps
\\
&
= 2 \sqrt{2\eps} \Big[ \tanh(u_\eps) \Big]_{-\de/\sqrt{2\eps}}^{\de/\sqrt{2\eps}} = 4 \sqrt{2\eps} \, \tanh\Big(\frac{\de}{\sqrt{2\eps}}\Big).
\end{aligned}
\end{equation}
From \eqref{eq:lunghezza_ricciolo} and \eqref{eq:eps_k_squared_ricciolo}, we obtain
\begin{equation}\label{eq:energia_alpha}
    F_\eps(\alpha_\eps,[-\de, \de])=
    2\de
    +
    4 \sqrt{2\eps} \, \tanh\Big(\frac{\de}{\sqrt{2\eps}}\Big).
\end{equation}
Therefore, combining \eqref{eq:energia_alpha}, \eqref{eq:energia_arco}, and \eqref{eq:peps-qeps}, we get
\begin{equation}\label{eq:energia_key_con}
\begin{aligned}
    G_\eps(\eta_\eps,[-\ell(\raccordo)-\de,\ell(\raccordo)+\de])&=
   \frac{2F_\eps(\raccordo)+F_\eps(\alpha_\eps,[-\de,\de])-|p_\eps-q_\eps|}{\eps^{1/2}}
    \\
    &
    = 
    \frac{2\ell(\raccordo)-2r_\eps \, \sin\theta_\eps(\de)}{\eps^{1/2}}
    +\sigma \, \tanh\Big(\frac{\de}{\sqrt{2\eps}}\Big)
    +2\eps^{1/2}\frac{\theta_\eps(-\de)}{r_\eps}
    \\
    &
    =:\textrm{I}_\eps+\textrm{II}_\eps+\textrm{III}_\eps.
\end{aligned}
\end{equation}
We want to pass to the limit as $\eps \to 0^+$
in \eqref{eq:energia_key_con}. Since we chose $\de$ as is \eqref{eq:scelta_de}, we immediately observe that 
\begin{equation*}\label{eq:limII}
    \lim_{\eps\to0^+} \textrm{II}_\eps=\sigma.
\end{equation*}
Moreover, we claim that both $\textrm{I}_\eps, \textrm{III}_\eps$ are negligible as $\eps \to 0^+$, i.e.
\begin{equation}\label{eq:limI}
   \lim_{\eps\to0^+} \textrm{I}_\eps=0,
\end{equation}
\begin{equation}\label{eq:limIII}
    \lim_{\eps\to0^+} \textrm{III}_\eps=0,
\end{equation}
so that, in conclusion,
\begin{equation}\label{eq:finale_energia}
    \lim_{\eps\to0^+} G_\eps(\keycon)=
    \lim_{\eps\to 0^+}(\textrm{I}_\eps+\textrm{II}_\eps+\textrm{III}_\eps)
    =\sigma.
\end{equation}
Concerning $\textrm{I}_\eps$, using the Taylor expansion $\sin(\theta)=\theta+o(\theta^2)$, we get
\begin{equation}\label{eq:I}
     \lim_{\eps\to0^+} \textrm{I}_\eps
     = \lim_{\eps\to0^+} \frac{2 r_\eps o(\theta^2_\eps(-\de))}{\eps^{1/2}}
     =  \lim_{\eps\to0^+} 2 r_\eps o(\theta_\eps(-\de)) \frac{o(\theta_\eps(-\de))}{\eps^{1/2}}.
\end{equation}
Let $$b=a-1/2<0.$$ Notice that by \eqref{eq:scelta_de}
\begin{equation}\label{eq:lim_e}
    \lim_{\eps\to 0^+} e^{-\eps^b/\sqrt{2}}=0.
\end{equation}
Substituting the expressions for $r_\eps$ and $\theta_\eps(-\de)$ from \eqref{eq:raggio} and \eqref{eq:angolo_tangente}, yields
\begin{equation}
\begin{aligned}\label{eq:I_bis}
\lim_{\eps\to0^+} r_\eps o(\theta_\eps(-\de))
& 
=\lim_{\eps\to 0^+} \sqrt{2\eps}
\frac{1+e^{-2\eps^b/\sqrt{2}}}{e^{-\eps^b/\sqrt{2}}} o(\arctan(e^{-\eps^b/\sqrt{2}}))
\\
&
=\lim_{\eps\to 0^+} \sqrt{2\eps}
\frac{1}{e^{-\eps^b/\sqrt{2}}} o(\arctan(e^{-\eps^b/\sqrt{2}}))
\\
&
=\lim_{\eps\to 0^+} \sqrt{2\eps}\frac{\arctan(e^{-\eps^b/\sqrt{2}})}{e^{-\eps^b/\sqrt{2}}}\frac{o(\arctan(e^{-\eps^b/\sqrt{2}}))}{\arctan(e^{-\eps^b/\sqrt{2}})}
\\
&=0
\end{aligned}
\end{equation}
and 
\begin{align}\label{eq:I_tris}
 \lim_{\eps\to0^+}  \frac{o(\theta_\eps(-\de))}{\eps^{1/2}}  
 \leq\lim_{\eps\to 0^+}\frac{o(\arctan(e^{-\eps^b/\sqrt{2}}))}{\arctan(e^{-\eps^b/\sqrt{2}})}\frac{e^{-\eps^b/\sqrt{2}}}{\eps^{1/2}}=0.
\end{align}
Hence, collecting \eqref{eq:I}, \eqref{eq:I_bis} and \eqref{eq:I_tris}, we get \eqref{eq:limI}.

Now, using \eqref{eq:angolo_tangente} and \eqref{eq:raggio}, we consider 
\begin{equation}\label{eq:III}
    \textrm{III}_\eps=\frac{32}{\sqrt{2}}\arctan(e^{-\eps^b/\sqrt{2}})\frac{e^{-\eps^b/\sqrt{2}}}{1+e^{-2\eps^b/\sqrt{2}}}.
\end{equation}
From \eqref{eq:lim_e} we have
\begin{equation}\label{eq:III_bis}
    \lim_{\eps\to 0^+} \arctan(e^{-\eps^b/\sqrt{2}})=0 \quad \text{ and } \quad \lim_{\eps\to 0^+} \frac{e^{-\eps^b/\sqrt{2}}}{1+e^{-2\eps^b/\sqrt{2}}}=0.
\end{equation}
Thus, from \eqref{eq:III} and \eqref{eq:III_bis}, we obtain \eqref{eq:limIII}.

\begin{Remark}\rm
The constant $8\sqrt{2}$ 
obtained as a result of the limit in \eqref{eq:energia_key_con}, 
can be compared with the value arising from a simpler construction 
consisting of a straight segment connecting $p$ and $q$, 
together with a circular loop of 
radius $\sqrt{\eps}$ attached to it. In this case, the total length is
\begin{equation*}
\ell(\gamma_\eps) = \vert p-q\vert + 2\pi \sqrt{\eps},
\end{equation*}
while the curvature contribution comes only from the circle and equals
\begin{equation*}
\eps \int_{\partial B_{\sqrt{\eps}}} \kappa_{\gae}^2 ds 
= \eps \frac{2\pi \sqrt{\eps}}{(\sqrt{\eps})^2}
= 2\pi \sqrt{\eps}.
\end{equation*}
Therefore,
\begin{equation*}
F_\eps(\gamma_\eps) = \vert p-q\vert+ 4\pi \sqrt{\eps},
\qquad
G_\eps(\gamma_\eps) = 4\pi.
\end{equation*}
Since $4\pi > 8\sqrt{2}$, this shows that $\keycon$ provides a more efficient recovery sequence.
\end{Remark}
\paragraph{Proof of the $\Gamma$-limsup inequality.}
We now prove the $\Gamma$-limsup inequality of Theorem \ref{thm:main_result}.

Let $(\Ga,\omega)\in \mathrm{Dom}_G$. By definition, $\Ga$ is the segment
$$
\Ga(x)=p+\frac{x}{\ex}(q-p)\qquad \text{for all }x\in[0,\ex],
$$
and
$$
\omega=2\pi\sum_{j=1}^N c_j\delta_{x_j},
\qquad c_j\in\mathbb Z,\quad 0\le x_1<\dots<x_N\le \ex.
$$
Set for simplicity
$$
M:=\sum_{j=1}^N |c_j|=G(\Ga,\omega)\geq N.
$$
We may rewrite $\omega$ as
$$
\omega=2\pi\sum_{h=1}^M \alpha_h\delta_{y_h},
$$
where $\alpha_h\in\{-1,1\}$, $y_h\in\{x_1,\dots,x_N\}$ and
$$
y_1\le y_2\le \dots\le y_M.
$$

Fix $a\in(\frac14,\frac12)$ and let $\delta_\eps$ be as in \eqref{eq:scelta_de}. For $\eps>0$ sufficiently small, we choose points
$$
\delta_\eps< y_1^\eps<y_2^\eps<\dots<y_M^\eps<\ex-\delta_\eps
$$
such that 
$$
y_{h+1}^\eps-y_h^\eps\ge 2\delta_\eps
\qquad \text{for every }h=1,\dots,M-1,
$$
and
$$
y_h^\eps\to y_h \qquad \text{as }\eps\to 0^+ \quad \text{for every }h=1,\dots,M.
$$

For each $h=1,\dots,M$, let $\eta_{\eps,h}$ denote the curve introduced in Paragraph~\ref{par:key_construction}.
If $\alpha_h=1$ we use the profile $\eta_\eps$, while if $\alpha_h=-1$ we use its reflection with respect to the $x$-axis, so that its curvature has opposite sign. In both cases the corresponding curve is of class $H^2$, has horizontal tangent at its endpoints, and, in view of \eqref{eq:finale_energia}, satisfies
$$
\lim_{\eps\to0^+} G_\eps(\eta_{\eps,h},[-\ell(\varsigma_{\eps})-\de,\ell(\raccordo)+\de])=\sigma.
$$

We now define a recovery sequence
$\Ga_\eps \in H^2([0,\pq];\mathbb{R}^2)$ as
\begin{equation*}
\Ga_\eps(\parameter):=
\begin{cases}
\eta_{\eps,h}\!\left(\dfrac{\parameter-y^\eps_h}{\pq}\,\lgae\right)
& \text{if } \parameter\in\left(y^\eps_h - \dfrac{\de}{\lgae}\pq,\; y^\eps_h + \dfrac{\de}{\lgae}\pq\right),\quad h=1,\dots,M,
\\[6pt]
\Ga(\parameter)
& \text{elsewhere}.
\end{cases}
\end{equation*}

Namely, in small intervals centered at $y^\eps_h$, we replace the segment $\Ga$ by suitably rescaled copies of $\eta_{\eps,h}$.

By construction, $\Ga_\eps(0)=p$, $\Ga_\eps(\pq)=q$. Moreover
\begin{equation}\label{eq:recovery_limit}
\lim_{\eps\to 0^+}
\bigl\| \kappa_{\Ga_\eps} - 2\pi\sum_{h=1}^M \alpha_h\delta_{y_h} \bigr\|_{\mathrm{flat},[0,\ex]} = 0.
\end{equation}

Indeed, following the same strategy as in Proposition~\ref{lem:compactness_in_flat_norm}, we fix $\radius>0$ and consider the disjoint intervals $(a_{\eps,\radius}^i,b_{\eps,\radius}^i)$, $i=1,\dots,M$, such that
$$
E_\radius^\eps
=
\bigcup_{i=1}^M (a_{\eps,\radius}^i,b_{\eps,\radius}^i).
$$
By construction of $\Ga_\eps$ and by definition of $E_\radius^\eps$ (see~\eqref{eq:E^eps_delta}), for each $i\in\{1,\cdots,M\}$ there exists $h$ such that
\begin{equation}\label{eq:ai}
a_{\eps,\radius}^i
\in
\left(
y^\eps_h - \frac{\de}{\lgae}\pq,\;
y^\eps_h + \frac{\de}{\lgae}\pq
\right).
\end{equation}
Since $\frac{\pq}{\lgae}\to 1$ and $y^\eps_h\to y_h$, it follows that, for any fixed $\radius$,
\begin{equation}\label{eq:ai_to_xi}
a_{\eps,\radius}^i \to y_h
\qquad \text{as }\eps\to 0^+.
\end{equation}
Combining this with Proposition~\ref{lem:compactness_in_flat_norm}, we obtain \eqref{eq:recovery_limit}.

It remains to show that 
$$
\limsup_{\eps\to 0^+} G_\eps(\Ga_\eps)
\leq 
\sigma G(\Ga,\omega).
$$
Since outside the intervals $\left(y^\eps_h - \frac{\de}{\lgae}\pq,\; y^\eps_h + \frac{\de}{\lgae}\pq\right)$ the curve is a straight segment, only the $M$ building blocks contribute to the energy. Therefore,
$$
\limsup_{\eps\to 0^+} G_\eps(\Ga_\eps)
\le \sum_{h=1}^M \limsup_{\eps\to 0^+} G_\eps(\eta_{\eps,h})
= \sigma M
= \sigma G(\Ga,\omega).
$$
This concludes the proof.
\section{The energy functionals in case of closed curves}\label{sec:energy_on_Xl}
We now turn our attention to the second problem concerning closed immersed plane curves.
For convenience, we begin by stating the following definitions. Let $\ell>0$.

\medskip

The class $\mathcal{B}_1$ is here defined as in Definition \ref{def:B_c} with $\ell$ replacing $\ex$. 

\begin{Definition}
Let $\ga\in\mathcal{B}_1$. We define the set of measures compatible with $\ga$ as 
\begin{equation*}
    \mathcal{K}(\ga):=\{\mu\in\mathcal{M}_b([0,\ell]):\mu=\partial_s\argf_{\ga}+\omega \text{ with } \omega\in M_{\text{fin},\mathbb Z}([0,\ell])\}.
\end{equation*}
\end{Definition}

\begin{Definition}
    A pair $(\ga,\mu)$ with $\ga\in\mathcal{B}_1$ and $\mu\in\mathcal{K}(\ga)$ is called a pointed curve.
\end{Definition}

In the present setting of closed curves, the test functions in the definition of the flat norm are required to be periodic. Namely, given a Radon measure $\mu$, we define
\begin{equation*}
\|\mu\|_{\text{flat},[0,\ell]} := \sup_{\substack{\varphi \in C^{0,1}([0,\ell]), \\ \|\varphi\|_{C^{0,1}([0,\ell])} \le 1 \\ \varphi(0)=\varphi(\ell)}}
\int_{[0,\ell]} \varphi \, d\mu \
.
\end{equation*}

The definition of convergence of a sequence of pointed curves $(\ga_n,\mu_n)$ to a pointed curve $(\ga,\mu)$ is as in Definition \ref{def:conv_B} with $\ex$ replaced by $\ell$.

\medskip

Let $\mathcal{T}=\{(\ga,\mu): \ga\in\mathcal{B}_1,
\mu\in\mathcal{K}(\ga) \}$ and
define 
$$X_\ell = 
\{
(\ga,\mu)\in\mathcal{T}:\ga \in H^2([0,\ell];\mathbb{R}^2), \ga(0)=\ga(\ell),\dot\ga(0)=\dot\ga(\ell), \ell(\ga)=\ell, \mu=\kappa_\ga
\},$$
namely, $\ga$ is a closed planar curve of fixed length $\ell$.

\nada{
Let $F_\eps:\mathcal{T}\to (-\infty,+\infty]$ defined as
\begin{equation*}
F_\eps(\ga,\mu)
=\begin{cases}
 \ell(\ga)+\eps\displaystyle\int_0^\ell \kappa_\ga^2 \,ds &
\text{ if }
(\ga,\mu) \in X_\ell, 
\\
+\infty & \text{ if } (\ga,\mu)\in\mathcal{T}\setminus X_\ell. 
\end{cases}
\end{equation*}
}
\begin{Definition}
We define by
$\rm Dom_G$ 
the set of all pairs $(\ga,\mu) \in\mathcal{T}$ such that
$\ga$ is a closed planar curve in $\mathcal B_1$ of length $\ell$.
\end{Definition}
An element of $\mathrm{Dom}_G$ will be called a pointed $\ell$-closed curve.

\begin{Definition}
We define $G:\mathcal{T} \to
[0, +\infty]$ 
the functional 
\begin{equation*}
G(\ga,\mu) := \begin{cases}
\sum_{j=1}^N \vert c_j\vert  & {\rm if }~ 
(\ga,\mu) \in {\rm Dom_G},
\\
+ \infty & {\rm if }~  \ga\in\mathcal{T}\setminus{\rm Dom_G}.
\end{cases}
\end{equation*}
\end{Definition}

\section{\texorpdfstring{$\Gamma$-limit for closed curves}{Gamma-limit for closed curves}}\label{sec:first_oder_exp_Xl}
\nada{
\textcolor{red}{Dire (forse nell'introduzione) che in questo caso il risultato non è esattamente di espansione in $\Gamma$ convergenza, ma un teorema del tipo: per ogni
successione che converge al cerchio in ..., vale la diseguaglianza del gammaliminf, ed esiste una
recovery sequence. 
}
Set 
$$
\begin{aligned}
G_\eps(\ga) := &
\frac{F_\eps(\ga)
- 
\min_{\ga\in X_{\ell}} F(\ga)}{\eps^{1/2}}
=
\frac{F_\eps(\ga)
- \ell
}{\eps^{1/2}}
=\eps^{1/2}\displaystyle\int_0^\ell \kappa_\ga^2 \,ds
\qquad (\ga,\kappa_\ga) \in 
{X_\ell}.
\end{aligned}
$$
}

Let $\ga\in C^2([0,\ell];\mathbb{R}^2)\cap \mathcal B_1$ be a closed curve of length $\ell$, with $\dot\ga(0)=\dot\ga(\ell)$ and $\ddot\ga(0)=\ddot\ga(\ell)$. Let $\theta=\theta_\ga\in C^0([0,\ell];\mathbb{R}^2)$ be a lifting of $\partial_s \ga$ and let $\kappa_\ga$ denote the curvature of $\ga$, i.e. $\kappa_\ga = \partial_s \theta$.

Set 
$$
\begin{aligned}
\mathcal{G}_\eps(\gae) 
:=\eps^{1/2}\displaystyle\int_0^\ell \kae^2 \,ds +\frac{1}{2\eps^{1/2}}\int_0^\ell |\partial_s\gae-\partial_s\ga|^2 \,ds
\qquad (\gae,\kae) \in {X_\ell}.
\end{aligned}
$$
A motivation for the second term in $\mathcal{G}_\eps$ is given by \eqref{eq:rate_conv}.

We focus on sequences $(\gae,\kae)$ in $X_\ell$ such that 
\begin{equation*}
    \mathcal{G}_\eps(\gae)\leq C \qquad \eps\in(0,1],
\end{equation*}
for some $C>0$,
which implies 
\begin{equation}\label{eq:closed_bounds}
 \eps^{1/2}\displaystyle\int_0^\ell \kae^2 \,ds\leq C,
 \qquad
 \int_0^\ell |\partial_s\gae-\partial_s\ga|^2 \,ds\leq C\eps^{1/2}.
\end{equation}
Notice that the second condition implies that $$
\gae\to\ga \text{ in } H^1([0,\ell];\mathbb{R}^2).$$

\begin{Remark}\label{rmk:Modica-Mortola}\rm
As in the case of open curves,
the functional $G_\eps$ has the structure of a one-dimensional
Modica--Mortola type energy.
Let $(\gamma_\eps,\kappa_{\gamma_\eps})\in X_\ell$, and let
$\theta_\eps:[0,\ell]\to\mathbb R$ be a lifting of the tangent vector of
$\gamma_\eps$, namely
$$
\partial_s\gamma_\eps(s)=\bigl(\cos\theta_\eps(s),\sin\theta_\eps(s)\bigr)
\qquad s\in[0,\ell].
$$
Since $\gamma_\eps\in H^2([0,\ell];\mathbb R^2)$ and $|\partial_s\gamma_\eps|=1$,
we also have
$
\kappa_{\gamma_\eps}=\partial_s\theta_\eps.
$
We now consider the second term in the definition of $\mathcal{G}_\eps$. Notice that
\begin{equation*}
\partial_s\gamma_\eps(s)-\partial_s\gamma(s)
=
\bigl(\cos\theta_\eps(s)-\cos\theta(s),\,
\sin\theta_\eps(s)-\sin\theta(s)\bigr),
\end{equation*}
and thus
\begin{align*}
|\partial_s\gamma_\eps(s)-\partial_s\gamma(s)|^2
&=
1+1-2\bigl(\cos\theta_\eps\cos\theta+\sin\theta_\eps\sin\theta\bigr)
\\
&=
2-2\bigl(\cos\theta_\eps\cos\theta+\sin\theta_\eps\sin\theta\bigr)
=
2\bigl(1-\cos(\theta_\eps(s)-\theta(s))\bigr).
\end{align*}

Substituting this identity into the definition of $\mathcal{G}_\eps$, we get
\begin{align*}
\mathcal{G}_\eps(\gamma_\eps)
=
\eps^{1/2}\int_0^\ell (\partial_s\theta_\eps)^2\,ds
+
\frac{1}{\eps^{1/2}}\int_0^\ell
\bigl(1-\cos(\theta_\eps-\theta)\bigr)\,ds.
\end{align*}
\end{Remark}

Our second main theorem reads as follows.
\begin{Theorem}[Compactness and $\Gamma$-convergence]\label{thm:main_result_2}
We have:
\begin{itemize}
\item[\textnormal{(i)}] Compactness and $\Gamma$-liminf inequality: if $\big((\ga_\eps,\kappa_{\ga_\eps})\big)\subset X_\ell$ is a sequence such that 
\eqref{eq:closed_bounds} holds, then,
up to a subsequence, there exist $\mu\in \mathcal{K}(\ga)$ such that $\big((\ga_\eps,\kappa_{\ga_\eps})\big)$ converges to $(\ga,\mu)$. Furthermore, if $\big((\ga_\eps,\kappa_{\ga_\eps})\big)$ converges to $(\ga,\mu)$ then 
%
$$
\liminf_{\eps\to 0^+} \mathcal{G}_\eps(\ga_\eps)
\geq 
\sigma G(\ga,\mu).
$$
\item[\textnormal{(ii)}] $\Gamma$-limsup inequality: 
for every $(\ga,\mu)\in \rm Dom_G$
there exists a sequence $\big((\ga_\eps,\kappa_{\ga_\eps})\big)\subset X_\ell$ converging to $(\ga,\mu)$ such that
$$
\lim_{\eps\to 0^+} \mathcal{G}_\eps(\ga_\eps)
=
\sigma G(\ga,\mu).
$$
\end{itemize}
\end{Theorem}
\nada{
\begin{Remark}\rm
The assumption \eqref{eq:necessaria} is needed to prevent
the first order $\Gamma$--limit of $G_\eps$ from becoming trivial.
For example, let us consider $\gae$ as follows.
Let $\gamma$ be the a circle parametrized by arclength on $[0,\ell]$.
Fix two parameters $\delta_\eps$ and $m_\eps$.
We modify $\gamma$ by adding
a small circular loop of radius $\de$ and multiplicity $m_\eps$. Denote by $c_\eps$ this loop and by $\gae$ the resulting curve. 
The excess length is compensated by a slight contraction of $\gamma$.
By direct computation
$$
G(c_\eps)
=
\eps^{1/2}\, m_\eps\,\frac{2\pi}{\de}.
$$
Choosing for instance
$
\de=\eps^{1/4},
m_\eps=\eps^{-1/8},
$
we obtain
$$
G(c_\eps)
=
C
\eps^{1/2}\,\frac{\eps^{-1/8}}{\eps^{1/4}}
=
\eps^{1/8},
$$
which hence tends to $0$.
Since also 
$\lim_{\eps\to 0} G(\gae\setminus c_\eps)=0$, we conclude that
\begin{equation}\label{eq:triviale}
    \lim_{\eps\to 0}G_\eps(\gae)=0.
\end{equation}
Notice that \eqref{eq:necessaria} cannot hold. Indeed, if it were satisfied, then \eqref{eq:kx_uniformly} would follow, contradicting \eqref{eq:triviale}.
\end{Remark}
}

The proof of Theorem \ref{thm:main_result_2} is split across Sections \ref{sec:compactness_2}, \ref{sec:lb_2}, and \ref{sec:up_2}.
\subsection{The compactness result}\label{sec:compactness_2}
In this section we prove the compactness part of Theorem~\ref{thm:main_result_2}. 
To this end, let $\big((\gamma_\eps,\kappa_{\gamma_\eps})\big)\subset X_\ell$ be a sequence such that \eqref{eq:closed_bounds} holds.
Define 
\begin{equation}\label{eq:phi_eps}
\phi_\eps:=\theta_\eps-\theta.
\end{equation}
Then
\begin{equation*}\label{eq:phi_prime_closed}
\partial_s\phi_\eps=\kappa_{\gamma_\eps}-\kappa_\gamma.
\end{equation*}
Fix $\radius\in(0,\frac12)$. We define 
\begin{equation}\label{eq:E_rho_closed}
E_\radius^\eps
:=
\left\{
s\in[0,\ell]:
|\partial_s\gamma_\eps(s)-\partial_s\gamma(s)|>\radius
\right\}.
\end{equation}
The second bound in \eqref{eq:closed_bounds} yields
\begin{equation}\label{eq:size_bad_set_closed}
|E_\radius^\eps|\,\radius^2
\le
\int_{E_\radius^\eps}|\partial_s\gamma_\eps-\partial_s\gamma|^2\,ds
\le
\int_0^\ell |\partial_s\gamma_\eps-\partial_s\gamma|^2\,ds
\le C\eps^{1/2}.
\end{equation}
Hence, arguing as in \eqref{eq:kx_uniformly} and using the first bound in \eqref{eq:closed_bounds},
\begin{align}
\int_{E_\radius^\eps} |\kappa_{\gamma_\eps}|\,ds
\le
\frac{C}{\radius}.
\label{eq:kappa_bad_set_closed}
\end{align}
 
Using \eqref{eq:kappa_bad_set_closed}, Cauchy--Schwarz inequality, $\kappa_\gamma\in L^2([0,\ell])$ and \eqref{eq:size_bad_set_closed} we also get
\begin{equation}\label{eq:phi_prime_bad_set_estimate}
\begin{aligned}
\int_{E_\radius^\epsilon} |\partial_s\phi_\epsilon|\,ds
\le
\int_{E_\radius^\epsilon} |\kappa_{\gamma_\epsilon}|\,ds
+
\int_{E_\radius^\epsilon} |\kappa_\gamma|\,ds 
\le
\frac{C}{\radius}
+
\|\kappa_\gamma\|_{L^2([0,\ell])}\,|E_\radius^\epsilon|^{1/2} \le
\frac{C}{\radius}
+
C\,\frac{\epsilon^{1/4}}{\radius}.
\end{aligned}
\end{equation}

\begin{Proposition}\label{lem:compactness_closed_flat}
Assume that \eqref{eq:closed_bounds} holds. Then there exist a subsequence $\eps_k$ and  $\omega\in M_{\rm fin,\mathbb Z}([0,\ell])$ such that
\begin{equation}\label{eq:flat_conv_closed}
\lim_{k\to +\infty}
\bigl\|
(\kappa_{\gamma_{\eps_k}}-\kappa_\gamma)-\omega
\bigr\|_{\mathrm{flat}}
=0.
\end{equation}
Equivalently,
\[
\kappa_{\gamma_{\eps_k}}\to \kappa_\gamma\,ds+\omega
\qquad\text{in the flat norm.}
\]
\end{Proposition}

\begin{proof}
By the $C^2$-regularity of $\ga$, we have 
$$E_\radius^\eps=\cup_{i=1}^\infty (a^i_{\eps,\radius},b_{\eps,\radius}^i).$$
Here we suppose without loss of generality, that all the intervals $(a^i_{\eps,\radius},b_{\eps,\radius}^i)$ do not contain the point $0$, although it cannot be excluded a priori, by the regularity of $\ga$. Anyway, in that case the arguments are exactly the same of the those that follow.
\smallskip

\noindent
\textit{Step 1.}
Fix one interval $(a_{\eps,\radius},b_{\eps,\radius}):=(a_{\eps,\radius}^i,b_{\eps,\radius}^i)\subset[0,\ell]$. From \eqref{eq:size_bad_set_closed} we deduce
\begin{equation*}
    |a_{\eps,\radius}-b_{\eps,\radius}|\leq C\frac{\eps^{1/2}}{\radius^2}.
\end{equation*}
Combined with the Lipschitz continuity of $\partial_s\ga$, this implies
\begin{equation}\label{eq:lip_ga}
 |\partial_s\ga(b_{\eps,\radius})-\partial_s\ga(a_{\eps,\radius})| \leq C\frac{\eps^{1/2}}{\radius^2}.
\end{equation}
As a consequence,
\begin{equation}\label{eq:thete}
 |\theta(b_{\eps,\radius})-\theta(a_{\eps,\radius})| \leq C\frac{\eps^{1/2}}{\radius^2}
\end{equation}
and, for all $s\in(a_{\eps,\radius},b_{\eps,\radius})$,
\begin{equation*}
|\partial_s\gae(s)-\partial_s\ga(a_{\eps,\radius})|
\geq 
|\partial_s\gae(s)-\partial_s\ga(s)|-|\partial_s\ga(s)-\partial_s\ga(a_{\eps,\radius})|
> \radius-C\frac{\eps^{1/2}}{\radius^2}>\radius-{\frac{\radius}{2}}=\frac{\radius}{2},
\end{equation*}
provided $\eps>0$ is small enough.
Hence, 
\begin{equation}\label{eq:theta_eps-theta}
    |\theta_\eps(s)-\theta(a_{\eps,\radius})|<2\pi-d_{\radius},
\end{equation}
where $d_{\radius}\geq 0$ is a function depending only on $\radius$ and vanishing as $\radius\to 0^+$.
Using \eqref{eq:theta_eps-theta} and \eqref{eq:thete} we get
\begin{equation}\label{eq:theta_eps-theta_x}
 |\theta_\eps(s)-\theta(s)|\le |\theta_\eps(s)-\theta(a_{\eps,\radius})|+|\theta(a_{\eps,\radius})-\theta(s)|\leq 2\pi-d_{\radius}+C\frac{\eps^{1/2}}{\radius^2} \quad  \forall s\in(a_{\eps,\radius},b_{\eps,\radius}).
\end{equation}
Notice that, by the definition of $E^\eps_\radius$ in \eqref{eq:E_rho_closed},
\begin{equation}\label{eq:estremi_intervallo}
    |\partial_s\gae(a_{\eps,\radius})-\partial_s\ga(a_{\eps,\radius})|=\radius= |\partial_s\gae(b_{\eps,\radius})-\partial_s\ga(b_{\eps,\radius})|.
\end{equation}
Now, combining \eqref{eq:estremi_intervallo} and \eqref{eq:lip_ga}, we obtain, for $\eps>0$ small enough,
\begin{equation*}
    \begin{aligned}
        |\partial_s\gae(a_{\eps,\radius})-\partial_s\gae(b_{\eps,\radius})|
        &\leq
         |\partial_s\gae(a_{\eps,\radius})-\partial_s\ga(a_{\eps,\radius})|+ |\partial_s\ga(a_{\eps,\radius})-\partial_s\ga(b_{\eps,\radius})|
         \\
         &\quad+ |\partial_s\ga(b_{\eps,\radius})-\partial_s\gae(b_{\eps,\radius})|
         \leq
         \radius+C\frac{\eps^{1/2}}{\radius^2}+\radius\leq 3\radius.
    \end{aligned}
\end{equation*}
Hence, by the general theory of liftings, there exists $k\in\mathbb Z$ such that
\begin{equation}\label{eq:theta_esp-theta_eps}
|\theta_\eps(a_{\eps,\radius})-\theta_\eps(b_{\eps,\radius})-2\pi k|
\le \hat d_{\radius},
\end{equation}
where $\hat d_{\radius}\geq 0$ is a function depending only on $\radius$ and vanishing as $\radius\to 0^+$. As a consequence
\begin{equation}\label{eq:2pik}
    |2\pi k|\leq \hat d_{\radius}+|\theta_\eps(a_{\eps,\radius})-\theta_\eps(b_{\eps,\radius})|\leq \hat d_{\radius}+2\pi+d_\radius.
\end{equation}
In particular, this yields $|k|\le 1$, hence
\begin{equation*}
k\in\{-1,0,1\}.
\end{equation*}

Let $\psi\in C^{0,1}([0,\ell])$ with
$\|\psi\|_{W^{1,\infty}}\le 1$ and $\psi(0)=\psi(\ell)$. Since $\partial_s\phi_\eps=\kappa_{\gamma_\eps}-\kappa_\gamma$, we have
\begin{equation}\label{eq:one_interval_start}
\begin{aligned}
\int_{a_{\eps,\radius}}^{b_{\eps,\radius}}(\kappa_{\gamma_\eps}-\kappa_\gamma)\psi\,ds
&=
\int_{a_{\eps,\radius}}^{b_{\eps,\radius}}\partial_s\phi_\eps\,\psi\,ds
=
(\psi(b_{\eps,\radius})\phi_\eps(b_{\eps,\radius})-\psi(a_{\eps,\radius})\phi_\eps(a_{\eps,\radius}))
-
\int_{a_{\eps,\radius}}^{b_{\eps,\radius}}\phi_\eps\partial_s\psi\,ds
\\
&=:\textrm{I}_{\eps,\radius}+\textrm{II}_{\eps,\radius}.
\end{aligned}
\end{equation}
From \eqref{eq:theta_eps-theta_x}, it follows that, on $(a_{\eps,\radius},b_{\eps,\radius})$,
\begin{equation}\label{eq:angolobis}
    |\phi_\eps|\leq2\pi
\end{equation}
and hence we obtain
\begin{equation}\label{eq:termineII}
    |\textrm{II}_{\eps,\radius}|
    \leq
    \bigg|\int_{a_{\eps,\radius}}^{b_{\eps,\radius}}\phi_\eps\partial_s\psi\,ds\bigg|
    \leq \|\partial_s\psi\|_{L^\infty}\int_{a_{\eps,\radius}}^{b_{\eps,\radius}}|\phi_\eps|\,ds
    \leq 
    2\pi\|\partial_s\psi\|_{L^\infty}(b_{\eps,\radius}-a_{\eps,\radius}).
\end{equation}
Now, we rewrite the boundary terms in \eqref{eq:one_interval_start} as 
\begin{equation}\label{eq:termine_I}
    \textrm{I}_{\eps,\radius}
    =
    \big(\psi(b_{\eps,\radius})-\psi(a_{\eps,\radius})\big)\phi_\eps(b_{\eps,\radius})+\psi(a_{\eps,\radius})\big(\phi_\eps(b_{\eps,\radius})-\phi_\eps(a_{\eps,\radius})\big)=:\textrm{I}_{\eps,\radius}^{(1)}+\textrm{I}_{\eps,\radius}^{(2)}
\end{equation}
and using \eqref{eq:angolobis} we estimate the first term as
\begin{equation}\label{eq:termineI_uno}
    |\textrm{I}_{\eps,\radius}^{(1)}|=
    |\big(\psi(b_{\eps,\radius})-\psi(a_{\eps,\radius})\big)\phi_\eps(b_{\eps,\radius})|\leq 2\pi\|\partial_s\psi\|_{L^\infty}(b_{\eps,\radius}-a_{\eps,\radius}).
\end{equation}
Now, we claim that
\begin{align*}
\bigl|\textrm{I}^{(2)}_{\eps,\radius}\mp 2\pi\,\psi(a_{\eps,\radius})\bigr|
=
\Bigl|
\psi(a_{\eps,\radius})
\bigl(\phi_\eps(b_{\eps,\radius})-\phi_\eps(a_{\eps,\radius})\mp 2\pi\bigr)
\Bigr|
\le
\|\psi\|_{L^\infty}\bigg(\hat d_{\radius}+C\frac{\eps^{1/2}}{\radius^2}\bigg).
\end{align*}

If $k=0$, from \eqref{eq:theta_esp-theta_eps} and \eqref{eq:thete},  we get
\begin{equation*}
|\phi_\eps(b_{\eps,\radius})-\phi_\eps(a_{\eps,\radius})|
=
|\theta_\eps(b_{\eps,\radius})-\theta_\eps(a_{\eps,\radius})-(\theta(b_{\eps,\radius})-\theta(a_{\eps,\radius}))|
\le \hat d_{\radius}+C\frac{\eps^{1/2}}{\radius^2}.
\end{equation*}
Hence
\begin{equation*}
|\textrm{I}_{\eps,\radius}^{(2)}|
\le
C\bigg(\hat d_{\radius}+\frac{\eps^{1/2}}{\radius^2}\bigg)\|\psi\|_{L^\infty}.
\end{equation*}
Combining this estimate with \eqref{eq:one_interval_start},\eqref{eq:termineII},\eqref{eq:termine_I},\eqref{eq:termineI_uno} we obtain
\begin{equation*}\label{eq:case_k0}
\left|
\int_{a_{\eps,\radius}}^{b_{\eps,\radius}}(\kappa_{\gamma_{\eps}}-\kappa_\gamma)\psi\,ds
\right|
\le
C\|\psi\|_{W^{1,\infty}}\bigg((b_{\eps,\radius}-a_{\eps,\radius})+\hat d_{\radius}+\frac{\eps^{1/2}}{\radius^2}\bigg).
\end{equation*}

If $k=\pm1$, from \eqref{eq:theta_esp-theta_eps} we get, respectively, 
\begin{equation*}
\big|\theta_\eps(b_{\eps,\radius})-\theta_\eps(a_{\eps,\radius})\mp 2\pi\big|
\le \hat d_{\radius},
\end{equation*}
and therefore, using \eqref{eq:thete}
\begin{equation*}
\big|\phi_\eps(b_{\eps,\radius})-\phi_\eps(a_{\eps,\radius})\mp2\pi\big|
\le
\hat d_{\radius}+C\frac{\eps^{1/2}}{\radius^2},
\end{equation*}
which concludes the proof of the claim.
Finally
\begin{align*}
\left|
\int_{a_{\eps,\radius}}^{b_{\eps,\radius}}
(\kappa_{\gamma_\eps}-\kappa_\gamma)\psi\,ds
\mp 2\pi\,\psi(a_{\eps,\radius})
\right|
\le
C\|\psi\|_{W^{1,\infty}}
\bigg((b_{\eps,\radius}-a_{\eps,\radius})+\hat d_{\radius}+\frac{\eps^{1/2}}{\radius^2}\bigg).
\end{align*}

In all cases, we conclude that
\begin{equation}\label{eq:estimate_interval_closed}
\left|
\int_{a_{\eps,\radius}}^{b_{\eps,\radius}}
\big((\kappa_{\gamma_\eps}-\kappa_\gamma)-\alpha_i 2\pi\delta_{a_{\eps,\radius}}\big)\psi\,ds
\right|
\le
C\|\psi\|_{W^{1,\infty}}\bigg((b_{\eps,\radius}-a_{\eps,\radius})+\hat d_{\radius}+\frac{\eps^{1/2}}{\radius^2}\bigg),
\end{equation}
where \(\alpha_i\in\{-1,0,1\}\).

\textit{Step 2.}
We now estimate the number of intervals \((a_{\eps,\radius}^i,b_{\eps,\radius}^i)\) for which the corresponding integer \(k\) is nonzero. Let \(N_{\eps,\radius}\) denote the number of such intervals.
From \eqref{eq:2pik} we infer
\begin{equation*}
    2\pi-\hat d_{\radius}\leq |\theta_\eps(a_{\eps,\radius})-\theta_\eps(b_{\eps,\radius})|,
\end{equation*}
and hence, using also \eqref{eq:thete}, we get 
\begin{equation*}
    |\phi_\eps(b_{\eps,\radius})-\phi_\eps(a_{\eps,\radius})|\geq2\pi-\hat d_{\radius}-C\frac{\eps^{1/2}}{\radius^2}.
\end{equation*}
Thus,
\begin{equation*}
2\pi-\hat d_{\radius}-C\frac{\eps^{1/2}}{\radius^2}
\le
\int_{a_{\eps,\radius}^i}^{b_{\eps,\radius}^i} |\partial_s\phi_\eps|\,ds.
\end{equation*}
Summing over all such intervals, we obtain
\begin{equation*}
N_{\eps,\radius}\bigg(2\pi-\hat d_{\radius}-C\frac{\eps^{1/2}}{\radius^2}\bigg)
\le
\int_{E_\radius^\eps} |\partial_s\phi_\eps|\,ds.
\end{equation*}
Using \eqref{eq:phi_prime_bad_set_estimate}, it follows that
\begin{equation*}
N_{\eps,\radius}\bigg(2\pi-\hat d_{\radius}-C\frac{\eps^{1/2}}{\radius^2}\bigg)
\le
\frac{C}{\radius}+C\frac{\eps^{1/4}}{\radius}.
\end{equation*}
Therefore, for $\radius=\frac{1}{8}$, there exists a constant $\hat C>0$ such that
\begin{equation*}
N_{\eps,\frac18}\le\hat C.
\end{equation*}
Moreover, if \(0<\radius'<\radius''<\frac18\), then
\begin{equation*}
N_{\eps,\radius'}\le N_{\eps,\radius''},
\end{equation*}
and so
\begin{equation*}
N_{\eps,\radius}\le\hat C
\qquad
\forall\,\radius\in\Big(0,\frac18\Big),
\end{equation*}
for all \(\eps\) sufficiently small.
Passing to a subsequence of $(\eps)$, we may therefore assume that there exists an integer \(N\ge0\) such that $N_{\eps,\radius}=N$ is constant and does not depend on $\eps$ and $\radius$.
Let $(a_{\eps,\radius}^1,b_{\eps,\radius}^1),\dots,(a_{\eps,\radius}^N,b_{\eps,\radius}^N)$
be the intervals corresponding to the case \(k=\pm1\). Set
\begin{equation*}
\omega_{\eps,\radius}:=2\pi\sum_{i=1}^{N} \alpha_i\,\delta_{a_{\eps,\radius}^i},
\end{equation*}
where $\alpha_i$ is $\pm1$ according to the case that $k=\pm1$. Notice carefully that $\omega_{\eps,\radius}\in M_{{\rm fin}, \mathbb Z}
([0,\ex])$ and that 
\begin{equation*}
    |\omega_{\eps,\radius}|([0,\ell])\leq 2\pi N.
\end{equation*}
By \eqref{eq:estimate_interval_closed}, summing over \(i=1,\dots,N\), we obtain
\begin{equation}\label{eq:estimate_bad_set_closed}
\left|
\int_{\cup_{i=1}^N(a_{\eps,\radius}^i,b_{\eps,\radius}^i)}
\bigl((\kappa_{\gamma_\eps}-\kappa_\gamma)-\omega_{\eps,\radius}\bigr)\psi\,ds
\right|
\le
C\|\partial_s\psi\|_{L^{\infty}}
\left(
\sum_{i=1}^N (b_{\eps,\radius}^i-a_{\eps,\radius}^i)+N\Big(\hat d_{\radius}+\frac{\eps^{1/2}}{\radius^2}\Big)
\right).
\end{equation}
Since
\begin{equation*}
\sum_{i=1}^N (b_{\eps,\radius}^i-a_{\eps,\radius}^i)\le |E_\radius^\eps|
\le C\frac{\eps^{1/2}}{\radius^2},
\end{equation*}
and \(N\) is independent of \(\eps\), we deduce from \eqref{eq:estimate_bad_set_closed} that
\begin{equation}\label{eq:estimate_bad_set_closed_final}
\left|
\int_{\cup_{i=1}^N(a_{\eps,\radius}^i,b_{\eps,\radius}^i)}
\bigl((\kappa_{\gamma_\eps}-\kappa_\gamma)-\omega_{\eps,\radius}\bigr)\psi\,ds
\right|
\le
C\|\partial_s\psi\|_{L^{\infty}}
\left(
\frac{\eps^{1/2}}{\radius^2}+\hat d_{\radius}
\right).
\end{equation}

\textit{Step 3.} It remains to estimate the flat norm of $\kae-\kappa_\ga$ on $[0,\ell]\setminus \bigcup_{i=1}^N (a_{\eps,\radius}^i,b_{\eps,\radius}^i)$. This is done exactly as in Step 3 of the proof of Proposition~\ref{lem:compactness_in_flat_norm} and yields 
\begin{equation*}
	\biggl|
	\int_{[0,\ell]\setminus \cup_{i=1}^N(a_{\eps,\radius}^i,b_{\eps,\radius}^i)}  (\kae-\kappa_\ga)\psi \, ds
	\biggr|
    \leq C\radius.
\end{equation*}
Finally, using \eqref{eq:estimate_bad_set_closed_final}, we arrive at
\begin{equation*}
	\biggl|
	\int_{0}^\ell  \bigl((\kappa_{\gamma_\eps}-\kappa_\gamma)-\omega_{\eps,\radius}\bigr)\psi\,ds
	\biggr|
	\leq C\radius
	+C\frac{ \eps^{1/2}}{\radius^2}+C|o_{\radius}(1)|\,,
\end{equation*}
for all $\psi\in C^{0,1}([0,\ell])$ with $\|\psi\|_{W^{1,\infty}}\leq 1$, $\psi(0)=\psi(\ell)$,
which shows that 
 for every fixed $\radius\in (0,\frac18)$,
\begin{equation*}
\limsup_{\eps\to 0^+}
\bigl\| (\kae-\kappa_\ga) - \omega_{\eps,\radius} \bigr\|_{\mathrm{flat},[0,\ell]}
\le C(\radius+o_{\radius}(1)).
\end{equation*}
Now, passing to a further subsequence, we may assume that
\begin{equation*}\label{eq:aeps_to_a_closed}
a_{\eps,\radius}^i\to a_\radius^i\in[0,\ell]
\qquad\text{for all }i=1,\dots,N.
\end{equation*}
Therefore, as $\eps\to 0^+$,
\begin{equation}\label{eq:muconv_bis}
\omega_{\eps,\radius}\to \omega_\radius:=2\pi\sum_{i=1}^{N} \alpha_i\delta_{a_\radius^i}
\qquad \qquad\text{ in the flat norm,}
\end{equation}
with 
\begin{equation}\label{eq:var_bis}
     |\omega_{\radius}|([0,\ell])\leq 2\pi N\leq C.
\end{equation}
Now up to a subsequence, by \eqref{eq:var_bis}, there exists $\omega\in M_{{\rm fin}, \mathbb Z}
([0,\ell])$ such that 
$$
\omega_\radius\to\omega \quad\text{ in the flat norm}, \qquad \text{as $\radius\rightarrow 0^+$}.$$
Combining the previous inequality  with \eqref{eq:muconv_bis}
we get
\begin{align*}
	\limsup_{\eps\to 0^+} &
	\big\| (\kae-\kappa_\ga) - \omega \bigr\|_{\mathrm{flat},[0,\ell]}
    \\
    &
    \leq
    \limsup_{\eps\to 0^+} \big(
	\bigl\| (\kae-\kappa_\ga) - \omega_{\eps,\radius} \bigr\|_{\mathrm{flat},[0,\ell]}+
	\bigl\| \omega_{\eps,\radius}-\omega_{\radius} \bigr\|_{\mathrm{flat},[0,\ell]}
    +
	\bigl\| \omega_{\radius}-\omega \bigr\|_{\mathrm{flat},[0,\ell]}\big)
    \\
    &\leq
    C(\radius+o_{\radius}(1))+
	\bigl\| \omega_{\radius}-\omega \bigr\|_{\mathrm{flat},[0,\ell]}.
\end{align*}
Since $\radius\in (0,\frac18)$ is arbitrary, we then obtain that
\begin{equation*}
\lim_{\eps\to 0^+}
\bigl\| (\kae-\kappa_\ga) - \omega \bigr\|_{\mathrm{flat},[0,\ell]} = 0.
\end{equation*}
\end{proof}

Let us define a strictly increasing function \(\Phi\in C^1(\mathbb R)\) by
\begin{equation}\label{eq:Phi_closed}
\Phi(r):=\int_0^r 2\sqrt{1-\cos t}\,dt .
\end{equation}

By an argument completely analogous to that used in the proof of Lemma~\ref{lem:theta_w}, we obtain

\begin{Lemma}\label{lem:Phi_phi_closed}
Let $\phi_\eps$ be as in \eqref{eq:phi_eps}.
We have
\begin{equation*}
\lim_{\eps\to0^+}
\left\|
\frac{2\pi}{8\sqrt2}\,\partial_s(\Phi\circ\phi_\eps)-\partial_s\phi_\eps
\right\|_{\mathrm{flat},[0,\ell]}
=0.
\end{equation*}
In particular, letting $\omega\in M_{{\rm fin},\mathbb Z}([0,\ex])$
be the measure for which \eqref{eq:flat_conv_closed} holds, we have
\begin{equation}\label{eq:Phi_phi_to_omega}
\frac{2\pi}{8\sqrt2}\,\partial_s(\Phi\circ\phi_\eps)\to \omega
\qquad\text{in the flat norm.}
\end{equation}
\end{Lemma}
\subsection{\texorpdfstring{$\Gamma$-liminf inequality}{Gamma-liminf inequality}}\label{sec:lb_2}
In this section we prove the lower bound inequality of Theorem \ref{thm:main_result_2}. To this purpose we may take, without loss of generality, a sequence $(\gae,\kae)$ such that
$$
 \eps^{1/2}\displaystyle\int_0^\ell \kae^2 \,ds\leq C,
 \qquad
 \int_0^\ell |\partial_s\gae-\partial_s\ga|^2 \,ds\leq C\eps^{1/2},
$$
which imply $\gae\to\ga$ in $H^1([0,\ell];\R^2)$ and that $\kae-\kappa_\ga$ converges in the flat norm to
$\omega=2\pi\sum_{j=1}^N c_j\,\delta_{x_j}$ with $N\geq 0$, $0\leq x_0<x_1<\cdots<x_N\leq\ell$ and $c_j\in \mathbb Z$.
It remains to prove that
\begin{equation}\label{eq:liminf_closed}
\liminf_{\eps\to0^+} \mathcal{G}_\eps(\gamma_\eps)
\ge
\sigma G(\ga,\omega).
\end{equation}

Recalling that $\phi_\eps=\theta_\eps-\theta$ and $\partial_s\phi_\eps=\kappa_{\gamma_\eps}-\kappa_\gamma$,  from Remark \ref{rmk:Modica-Mortola} we have that
$$
\mathcal{G}_\eps(\gamma_\eps)
=
\sqrt\eps\int_0^\ell (\partial_s\phi_\eps+\kappa_\gamma)^2\,ds
+
\frac1{\sqrt\eps}\int_0^\ell (1-\cos\phi_\eps)\,ds.
$$
Let $\eta\in(0,1)$. By Young's inequality $
2ab\ge -\eta a^2-\frac1\eta b^2
$,  we obtain
$$
(\partial_s\phi_\eps+\kappa_\gamma)^2
= (\partial_s\phi_\eps)^2+2\partial_s\phi_\eps\,\kappa_\ga+\kappa_\ga^2
\ge
(1-\eta)(\partial_s\phi_\eps)^2
-\Bigl(\frac1\eta-1\Bigr)\kappa_\gamma^2.
$$
Therefore
$$
\mathcal{G}_\eps(\gamma_\eps)
\ge
\sqrt\eps(1-\eta)\int_0^\ell |\partial_s\phi_\eps|^2\,ds
+
\frac1{\sqrt\eps}\int_0^\ell (1-\cos\phi_\eps)\,ds
-
C_\eta\sqrt\eps,
$$
where $C_\eta:=\Bigl(\frac1\eta-1\Bigr)\displaystyle\int_0^\ell \kappa_\gamma^2\,ds$. 
Since $\kappa_\gamma\in L^2((0,\ell))$, we have $C_\eta\sqrt\eps=o(1)$ as
$\eps\to0^+$.

Using once more the algebraic inequality
$
\sqrt{\eps}\,a^2+\frac{1}{\sqrt{\eps}}\,b^2\ge 2|a||b|,
$
with
$
a=\sqrt{1-\eta}\,|\partial_s\phi_\eps|,
b=\sqrt{1-\cos\phi_\eps},
$
we deduce
\begin{align*}
\mathcal{G}_\eps(\gamma_\eps)
&\ge
\sqrt{1-\eta}\int_0^\ell 2|\partial_s\phi_\eps|\sqrt{1-\cos\phi_\eps}\,ds
-o(1)
\\
&=
\sqrt{1-\eta}
\int_0^\ell |\partial_s(\Phi\circ\phi_\eps)|\,ds-o(1)
=
\sqrt{1-\eta}|\partial_s(\Phi\circ\phi_\eps)|([0,\ell])-o(1),
\end{align*}
where $\Phi$ is defined in \eqref{eq:Phi_closed}. Hence the sequence $(\Phi\circ\phi_\eps)_\eps$ is uniformly bounded in $BV((0,\ell))$. Therefore, up to a further subsequence, there exists $u\in BV((0,\ell))$ such that
$$
\Phi\circ\phi_\eps \to u
\qquad \text{in }L^1((0,\ell)).
$$
Hence the sequence $\partial_s(\Phi\circ\phi_\eps)$ converges weakly* to $\partial_su$ and,
by the lower semicontinuity of the total variation with respect to $L^1$ convergence, we obtain
\begin{equation*}
\liminf_{\eps\to0^+} \mathcal{G}_\eps(\gamma_\eps)
\ge
\sqrt{1-\eta}\,
|\partial_s u|([0,\ell]).
\end{equation*}
Since $\eta\in(0,1)$ is arbitrary, letting $\eta\to0^+$ gives
\begin{equation}\label{eq:lsc_closed}
\liminf_{\eps\to0^+} \mathcal{G}_\eps(\gamma_\eps)
\ge
|\partial_s u|((0,\ell)).
\end{equation}

It remains to identify the limit measure $\partial_s u$.
By Lemma~\ref{lem:Phi_phi_closed} we deduce 
\begin{equation*}\label{eq:Phi_phi_to_omega_closed}
\partial_s u=\frac{\sigma}{2\pi}\omega=\sigma\sum_{j=1}^N c_j\delta_{x_j}.
\end{equation*}
Therefore, 
\begin{equation}\label{eq:Du_total_closed}
    |\partial_s u|([0,\ell])=\sigma\sum_{j=1}^N |c_j|.
\end{equation}
From \eqref{eq:lsc_closed} and \eqref{eq:Du_total_closed} we get
\begin{equation*}
\liminf_{\eps\to0^+} \mathcal{G}_\eps(\gamma_\eps)
\ge
\sigma\sum_{j=1}^N |c_j|=\sigma G(\ga,\omega).
\end{equation*}
\subsection{\texorpdfstring{$\Gamma$-limsup inequality}{Gamma-limsup inequality}}\label{sec:up_2}

In this section we establish the upper bound in Theorem~\ref{thm:main_result_2}. 
We start with the following preliminary lemma.

\begin{Lemma}\label{lem:opposite_tangent_point}
Let $\gamma\in C^2([0,\ell];\R^2)$ be a closed regular curve parametrized by arclength with $\dot\ga(0)=\dot\ga(\ell)$, $\ddot\ga(0)=\ddot\ga(\ell)$, and let
$\theta_\gamma\in([0,\ell];\mathbb R)$ be a continuous lifting of $\partial_s\ga$ so that $\theta_\ga(0)=0, \theta_\ga(\ell)= 2\pi$.
Then, for every $s_1\in[0,\ell)$, there exists $s_2\in[0,\ell)$ such that
$$
|\theta_\gamma (s_2)-\theta_\gamma(s_1)|=\pi+2\pi k
$$
for some $k\in\mathbb Z$.
Equivalently,
$$\partial_s\gamma(s_2)=-\partial_s\gamma(s_1).$$
\end{Lemma}
\begin{proof}
Fix $s_1\in[0,\ell)$. Define the continuous function $\hat\theta_\gamma:[s_1,s_1+\ell]\to\R$ by
\begin{equation*}
\hat \theta_\ga(s)=
\begin{cases}
    \theta_\ga(s) &\text{if } s\in[s_1,\ell] \\
    \theta_\ga(s-\ell)-2\pi &\text{if } s\in(\ell,\ell+s_1].
\end{cases}
\end{equation*}
Notice that $\hat\theta_\ga(s_1)=\theta_\ga(s_1)$ and $\hat\theta_\ga(\ell+s_1)=\theta_\ga(s_1)+2\pi$. Hence, by continuity, there exists $\hat s_2\in(s_1,s_1+\ell]$ such that 
\begin{equation}\label{eq:continuità}
    \hat\theta_\ga(\hat s_2)=\theta_\ga(s_1)+\pi.
\end{equation}
Now define
\begin{equation*}
s_2:=
\begin{cases}
\hat s_2 & \text{if } \hat s_2\in[s_1,\ell],\\[4pt]
\hat s_2-\ell & \text{if } \hat s_2\in(\ell,s_1+\ell].
\end{cases}
\end{equation*}
If $\hat s_2\in(s_1,\ell]$, then from \eqref{eq:continuità}
$$
\theta_\gamma(s_2)-\theta_\gamma(s_1)
= \hat\theta_\ga(\hat s_2)-\theta_\gamma(s_1)
=
\pi.
$$
If instead $\hat s_2\in(\ell,s_1+\ell]$, then $s_2=\hat s_2-\ell$ and 
$$\theta_\gamma(s_2)-\theta_\gamma(s_1)=
\theta_\ga(\hat s_2-\ell)-\theta_\gamma(s_1)=
\hat\theta_\ga(\hat s_2)+2\pi-\theta_\gamma(s_1)=3\pi.
$$
In both cases,
$$\theta_\gamma(s_2)-\theta_\gamma(s_1)\in \pi+2\pi\mathbb Z.$$
\end{proof}

We are now in a position to construct the recovery sequence.
\paragraph{Upper bound for one curvature singularity.}
We begin by considering the case of just one curvature singularity. 
Fix a point $s_1\in[0,\ell)$. 
By Lemma~\ref{lem:opposite_tangent_point}, there exists a point $s_2\in[0,\ell)$ such that
$$
\partial_s\gamma(s_2)=-\partial_s\gamma(s_1).
$$

Let $\keycon:[s_1,s_1+\ell(\keycon)]\to\R^2$ be the curve constructed in Section~\ref{par:key_construction}. Namely, if $\de$ is as in \eqref{eq:scelta_de},
then $\keycon$ is obtained by taking the borderline elastica $\alpha_\eps$ restricted to an interval of length $2\de$ and attaching at its endpoints two circular connecting arcs $\raccordo$.
Each arc is extended until the tangent becomes horizontal.

Denote by $R\in SO(2)$ the rotation satisfying 
$$
R\,\partial_s{\keycon}(s_1)=\partial_s\gamma(s_1).
$$
We define the rotated building block of the key construction 
$$
\hat{\keycon}(s):=
\gamma(s_1)+R\bigl(\keycon(s)-\keycon(s_1)\bigr),
\qquad s\in[s_1,s_1+\ell(\keycon)].
$$
In this way
$$
\hat{\keycon}(s_1)=\gamma(s_1),
\qquad
\partial_s\hat{\keycon}(s_1)=\partial_s\gamma(s_1).
$$

Moreover, let $p_\eps$ and $q_\eps$ be such that
$$
p_\eps=\hat\keycon(s_1), 
\qquad 
q_\eps=\hat{\keycon}(s_1+\ell(\hat\keycon)).
$$

Let $\Sigma_\epsilon$ be the straight segment of length $|p_\eps-q_\eps|$, parametrized by arclength on $[0,|p_\eps-q_\eps|]$, such that
$$
\partial_s\Sigma_\epsilon=\partial_s\gamma(s_2).
$$

We now define a curve 
$$
\trec:[0,\ell(\trec)]\to\R^2,
\qquad
\ell(\trec)=\ell+\ell(\keycon)+|p_\eps-q_\eps|,
$$
by concatenation as follows:
\begin{equation*}
\trec(s)
=
\begin{cases}
\gamma(s)
& \text{for } s\in[0,s_1],
\\[6pt]

\hat\keycon(s)
& \text{for } s\in[s_1,s_1+\ell(\hat\keycon)],
\\[6pt]

\gamma\bigl(s-\ell(\hat\keycon)\bigr)
& \text{for } s\in[s_1+\ell(\hat\keycon),\,s_2+\ell(\hat\keycon)],
\\[6pt]

\Sigma_\epsilon\bigl(s-(s_2+\ell(\hat\keycon))\bigr)
& \text{for } s\in[s_2+\ell(\hat\keycon),\,s_2+\ell(\hat\keycon)+|p_\eps-q_\eps|],
\\[6pt]

\gamma\bigl(s-\ell(\hat\keycon)-|p_\eps-q_\eps|\bigr)
& \text{for } s\in[s_2+\ell(\hat\keycon)+|p_\eps-q_\eps|,\,\ell(\trec)].
\end{cases}
\end{equation*}
One checks that $\tilde\recovery\in H^2([0,\ell(\tilde\gae)];\mathbb R^2)$, see for instance \cite[Lemma 4.1]{BeMu04}. Furthermore, for $\eps>0$ sufficiently small, $\tilde\recovery$ has exactly one self-intersection, due to the elastica $\alpha_\eps$.
From \eqref{eq:lunghezza_ricciolo} and \eqref{eq:lunghezza_raccordo} we get 
\begin{equation}
\begin{aligned}\label{eq:lunghezza_keycon}
\ell(\keycon)
&=\ell(\alpha_\eps)+2\ell(\raccordo)=
2\de
+2\sqrt{2\eps}
\frac{1+e^{-2\de/\sqrt{2\eps}}}{e^{-\de/\sqrt{2\eps}}}
\arctan (e^{-\de/\sqrt{2\eps}})
\\
&\leq
2\de
+2\sqrt{2\eps}
(1+e^{-2\de/\sqrt{2\eps}})
\leq 2\de
+4\sqrt{2\eps}
\leq C\de,
\end{aligned}
\end{equation}
where we used $\arctan(x)\leq x$ for $x\geq 0$ and \eqref{eq:scelta_de}. Recalling \eqref{eq:peps-qeps} and using $\sin(x)\leq x$ for $x\geq0$, we also obtain
\begin{equation}
\begin{aligned}\label{eq:differenza_endpoints}
|p_\eps-q_\eps|
&=
2\de-4\sqrt{2\eps}\,
\tanh\!\left(\frac{\de}{\sqrt{2\eps}}\right)+2r_\eps \, \sin\theta_\eps(\de)
\\
&\leq 2\de+2 r_\eps\theta_\eps
= 2\de
+2\ell(\raccordo)
\leq 
C\de,
\end{aligned}
\end{equation}
where the last inequality follows from \eqref{eq:lunghezza_keycon}. Thus, combining \eqref{eq:lunghezza_keycon} and \eqref{eq:differenza_endpoints}, we obtain
\begin{equation}\label{eq:differenza}
    \ell(\tilde\recovery)-\ell
    =\ell(\keycon)+|p_\eps-q_\eps|
    \leq C\de.
\end{equation} 

Let now $\lambda_\epsilon:=\frac{\ell}{\ell(\tilde\recovery)}$ and define the recovery sequence by the homothety
$$
\recovery(s):=\lambda_\epsilon\,\tilde\recovery\!\left(\frac{s}{\lambda_\epsilon}\right),
\qquad s\in[0,\ell].
$$
Specifically,
\begin{equation*}
\gamma_\eps(s)
=
\begin{cases}
\lambda_\eps \gamma\!\left(\frac{s}{\lambda_\eps}\right)
& \text{for } s\in[0,\lambda_\eps s_1]=:I_1,
\\

\lambda_\eps \hat\keycon\!\left(\frac{s}{\lambda_\eps}\right)
& \text{for } s\in[\lambda_\eps s_1,\lambda_\eps(s_1+\ell(\hat\keycon))]=:I_2,
\\

\lambda_\eps \gamma\!\left(\frac{s}{\lambda_\eps}-\ell(\hat\keycon)\right)
& \text{for } s\in[\lambda_\eps(s_1+\ell(\hat\keycon)),
\lambda_\eps(s_2+\ell(\hat\keycon))]=:I_3,
\\

\lambda_\eps \Sigma_\eps\!\left(\frac{s}{\lambda_\eps}-(s_2+\ell(\hat\keycon))\right)
& \text{for } s\in[\lambda_\eps(s_2+\ell(\hat\keycon)),
\lambda_\eps(s_2+\ell(\hat\keycon)+|p_\eps-q_\eps|)]=:I_4,
\\

\lambda_\eps \gamma\!\left(\frac{s}{\lambda_\eps}-\ell(\hat\keycon)-|p_\eps-q_\eps|\right)
& \text{for } s\in[\lambda_\eps(s_2+\ell(\hat\keycon)+|p_\eps-q_\eps|),
\ell]=:I_5.
\end{cases}
\end{equation*}
Recalling \eqref{eq:differenza}, for $\eps>0$ sufficiently small, $\recovery$ has exactly one self-intersection, due to the elastica. One can also check that $\recovery$ is closed and that $\ell(\recovery)=\ell$.
Moreover, as $\eps\to0^+$, the lengths of the curves $\hat{\keycon}$ and $\Sigma_\eps$ tend to zero, and $\lambda_\eps\to1$. It follows that $\recovery\to\gamma$ pointwise on $[0,\ell]$.
We next prove that the convergence is in fact strong in $H^1([0,\ell];\R^2)$. More precisely, there exists a constant $C>0$, independent of $\eps$, such that
\begin{equation}\label{eq:conv_H1_recovery}
\|\partial_s\recovery-\partial_s\gamma\|_{L^2(0,\ell)}^2\le C\eps^{1/2}.
\end{equation}

Let $s\in I_1$. By definition,
$$
\partial_s\gamma_\eps(s)=\partial_s\gamma\!\left(\frac{s}{\lambda_\eps}\right).
$$
Since $\gamma\in C^2([0,\ell];\R^2)$, the map $\partial_s\gamma$ is Lipschitz continuous on $[0,\ell]$. Therefore
$$
|\partial_s\gamma_\eps(s)-\partial_s\gamma(s)|^2
\le
C\left|\frac{s}{\lambda_\eps}-s\right|^2.
$$
Using $\lambda_\eps=\ell/\ell(\tilde\gamma_\eps)$, we compute
$$
\frac{s}{\lambda_\eps}-s
=
s\left(\frac{1}{\lambda_\eps}-1\right)
=
\frac{s}{\ell}\bigl(\ell(\tilde\gamma_\eps)-\ell\bigr).
$$
Hence
$$
\int_{I_1}|\partial_s\gamma_\eps(s)-\partial_s\gamma(s)|^2\,ds
\le
\frac{C}{\ell^2}|\ell(\tilde\gamma_\eps)-\ell|^2
\int_{I_1}s^2\,ds
\leq 
\frac{C}{\ell^2}|\ell(\tilde\gamma_\eps)-\ell|^2.
$$
Using
\eqref{eq:differenza} we obtain
$$
\int_{I_1}|\partial_s\gamma_\eps(s)-\partial_s\gamma(s)|^2\,ds
\le C\de^2.
$$
Since $\de=\eps^a$ with $a\in(\frac14,\frac12)$, it follows that
$$
\int_{I_1}|\partial_s\gamma_\eps(s)-\partial_s\gamma(s)|^2\,ds
\le C\eps^{1/2},
$$
and
\begin{equation}\label{eq:limite_uno}
\frac{1}{2\eps^{1/2}}
\int_{I_1}|\partial_s\gamma_\eps(s)-\partial_s\gamma(s)|^2\,ds
\le
C\,\eps^{2a-\frac12}\to0
\qquad\text{as }\eps\to0^+.
\end{equation}
The same argument applies to the intervals $I_3$ and $I_5$, that is, to all the intervals where $\recovery$ is given by a reparametrization of $\gamma$.

\medskip

Let $s\in I_4$; here $\recovery(s)$ is the straight segment connecting the two
endpoints of the adjacent arcs and satisfying
$$
\partial_s\gamma(s_2)=\partial_s\Sigma_\epsilon.
$$
Thus,
\begin{align*}
|\partial_s\recovery(s)-\partial_s\gamma(s)|^2
&=
\left|
\partial_s\Sigma_\eps\!\left(\tfrac s{\lambda_\eps}-(s_2+\ell(\hat\keycon))\right)
-\partial_s\gamma(s)
\right|^2
\\
&=
\left|
\partial_s\gamma(s_2)-\partial_s\gamma(s)
\right|^2
\le
C\,|s-s_2|^2.
\end{align*}
Since $s\in I_4$ and both $\ell(\hat\keycon)$ and $|p_\eps-q_\eps|$ are of order $\de$ (see \eqref{eq:lunghezza_keycon} and \eqref{eq:differenza_endpoints}), it follows that
$$
|s-s_2|\le C\de
\qquad\text{for all } s\in I_4.
$$
Hence
$$
|\partial_s\recovery(s)-\partial_s\gamma(s)|^2\le C\de^2
\qquad\text{for all } s\in I_4.
$$
Therefore, 
\begin{align*}
\int_{I_4}
|\partial_s\recovery(s)-\partial_s\gamma(s)|^2\,ds
&\le C\de^2 |I_4|
\le C\de^2 |p_\eps-q_\eps|
\le C\de^3.
\end{align*}
From \eqref{eq:scelta_de}, we obtain
$$
\int_{I_4}
|\partial_s\recovery(s)-\partial_s\gamma(s)|^2\,ds
\le C\de^3
\le C\eps^{1/2},
$$
and 
\begin{equation}\label{eq:limite_due}
\frac{1}{2\eps^{1/2}}
\int_{I_4}|\partial_s\recovery(s)-\partial_s\gamma(s)|^2 \,ds
\le
C\,\eps^{3a-\frac12}\to0
\end{equation}
as $\eps\to 0^+$.

\medskip

We finally consider the case in which $s\in I_2$, where $\recovery(s)$ coincides with the curve constructed in Section~\ref{par:key_construction}.
Let
$$
v:=\partial_s\gamma(s_1)=\partial_s\hat\keycon(s_1).
$$
By the change of variable $t=s/\lambda_\eps$, we have
\begin{equation}\label{eq:pezzo_I2}
\int_{I_2}|\partial_s\recovery(s)-\partial_s\gamma(s)|^2\,ds
=
\lambda_\eps
\int_{s_1}^{s_1+\ell(\hat\keycon)}
\left|
\partial_s\hat\keycon(t)-\partial_s\gamma(\lambda_\eps t)
\right|^2\,dt.
\end{equation}
From the Lipschitz continuity of $\partial_s\gamma$ and \eqref{eq:lunghezza_keycon}, we have
$$
|\partial_s\gamma(\lambda_\eps t)-v|
\leq C|\lambda_\eps t-s_1|
\leq C\de,
\qquad t\in[s_1,s_1+\ell(\hat\keycon)].
$$

Therefore
\begin{equation}\label{eq:interno_I2}
\begin{aligned}
\left|
\partial_s\hat\keycon(t)-\partial_s\gamma(\lambda_\eps t)
\right|^2
&\le
2|\partial_s\hat\keycon(t)-v|^2
+2|v-\partial_s\gamma(\lambda_\eps t)|^2
\\
&\le
2|\partial_s\hat\keycon(t)-v|^2+C\de^2.
\end{aligned}
\end{equation}
Hence, from \eqref{eq:pezzo_I2} and \eqref{eq:interno_I2}, we obtain
\begin{equation}\label{eq:caseC_reduction}
\int_{I_2}|\partial_s\recovery(s)-\partial_s\gamma(s)|^2\,ds
\leq
C
\int_{s_1}^{s_1+\ell(\hat\keycon)}
|\partial_s\hat\keycon(t)-v|^2\,dt
+C\de^2.
\end{equation}
Now, by construction, $\hat\keycon$ is obtained from $\keycon$ by a rigid motion, hence
$$
\int_{s_1}^{s_1+\ell(\hat\keycon)}
|\partial_s\hat\keycon(t)-v|^2\,dt
=
\int_0^{\ell(\keycon)}
|\partial_s\keycon(\sigma)-e_1|^2\,d\sigma.
$$
We split the latter integral into the contribution of the borderline elastica
$\alpha_\eps$ and the two circular connecting arcs $\raccordo$:
\begin{equation}\label{eq:caseC_split}
\int_0^{\ell(\keycon)}
|\partial_s\keycon-e_1|^2\,d\sigma
=
\int_{-\de}^{\de}|\partial_s\alpha_\eps-e_1|^2\,ds
+
2\int_{\raccordo}|\partial_s\raccordo-e_1|^2\,ds.
\end{equation}

We first consider the elastica part. Since
$$
\partial_s\alpha_\eps=(\cos\theta_\eps,\sin\theta_\eps),
$$
we compute
$$
|\partial_s\alpha_\eps-e_1|^2
=
(\cos\theta_\eps-1)^2+\sin^2\theta_\eps
=
2(1-\cos\theta_\eps).
$$
For the borderline elastica one has the equipartition identity
$$
1-\cos\theta_\eps=\eps\,\kappa_{\alpha_\eps}^2.
$$
Indeed, setting \(u=e^{s/\sqrt{2\eps}}\), by \eqref{eq:angolo_tangente} and \eqref{eq:curvatura},
\begin{equation*}
\begin{aligned}
1-\cos\theta_\eps
&=1-\cos(4\arctan u) 
=2\sin^2(2\arctan u) 
=2\left(\frac{2u}{1+u^2}\right)^2 
=\frac{8u^2}{(1+u^2)^2} \\
&
=\eps\left(\frac{4}{\sqrt{2\eps}}\frac{u}{1+u^2}\right)^2 
=\eps\,\kappa_{\alpha_\eps}^2.
\end{aligned}
\end{equation*}
Therefore, from \eqref{eq:eps_k_squared_ricciolo},
\begin{equation}\label{eq:caseC_elastica}
\int_{-\de}^{\de}|\partial_s\alpha_\eps-e_1|^2\,ds
=
2\eps\int_{-\de}^{\de}\kappa_{\alpha_\eps}^2\,ds
= 8 \sqrt{2\eps} \, \tanh\!\left(\frac{\de}{\sqrt{2\eps}}\right).
\end{equation}

We now estimate the contribution of each connecting arc. Parametrizing the arc by
the turning angle $t\in[0,\theta_\eps(-\de)]$, we have
$$
\partial_s\raccordo=(\cos t,\sin t),
\qquad
ds=r_\eps\,dt.
$$
Therefore
$$
\int_{\raccordo}|\partial_s\raccordo-e_1|^2\,ds
=
2r_\eps\int_0^{\theta_\eps(-\de)}(1-\cos t)\,dt
=
2r_\eps\bigl(\theta_\eps(-\de)-\sin(\theta_\eps(-\de))\bigr).
$$
We have
$$
\theta_\eps(-\de)-\sin(\theta_\eps(-\de))=O(\theta_\eps(-\de)^3)
$$
where, by \eqref{eq:angolo_tangente},
$$
\theta_\eps(-\de)=4\arctan(e^{-\de/\sqrt{2\eps}})
\le 4e^{-\de/\sqrt{2\eps}},
$$
while by \eqref{eq:raggio} one has $r_\eps=O(\sqrt\eps\,e^{\de/\sqrt{2\eps}})$.
Hence
$$
r_\eps\,\theta_\eps(-\de)^3
=
O\!\left(\sqrt\eps\,e^{-2\de/\sqrt{2\eps}}\right)
=
o(\sqrt\eps).
$$
Therefore each circular arc gives a contribution of order $o(\sqrt\eps)$ to the energy. Combining this with
\eqref{eq:caseC_split} and \eqref{eq:caseC_elastica}, we obtain
\begin{equation}\label{eq:caseC_key_final}
\int_0^{\ell(\keycon)}
|\partial_s\keycon-e_1|^2\,d\sigma
=
8\sqrt{2\eps}
\tanh\!\left(\frac{\de}{\sqrt{2\eps}}\right)
+
o(\sqrt\eps).
\end{equation}

Finally, inserting \eqref{eq:caseC_key_final} into \eqref{eq:caseC_reduction} and using
$\lambda_\eps\to1$, we conclude that
$$
\int_{I_2}|\partial_s\recovery(s)-\partial_s\gamma(s)|^2\,ds
=
8\sqrt{2\eps}\,
\tanh\!\left(\frac{\de}{\sqrt{2\eps}}\right)
+
o(\sqrt\eps).
$$
Thus we get
$$
\int_{I_2}|\partial_s\recovery(s)-\partial_s\gamma(s)|^2\,ds\leq C\eps^{1/2},
$$
and 
\begin{equation}\label{eq:limite_tre}
\lim_{\eps\to 0^+}
\frac{1}{2\eps^{1/2}}
\int_{I_2}|\partial_s\recovery(s)-\partial_s\gamma(s)|^2 \,ds
= 4\sqrt{2}.
\end{equation}
Collecting the estimates obtained on the intervals $I_1,\dots,I_5$, we obtain \eqref{eq:conv_H1_recovery}.

\medskip

Now we compute the energy of the recovery sequence. We first estimate the curvature term.
Recalling that
$$
\recovery(s)=\lambda_\eps\,\tilde\recovery\!\left(\frac{s}{\lambda_\eps}\right),
\qquad
\lambda_\eps=\frac{\ell}{\ell(\tilde\recovery)},
$$
a direct computation shows that
$$
\kappa_{\recovery}(s)
=
\frac{1}{\lambda_\eps}\,
\kappa_{\tilde\recovery}\!\left(\frac{s}{\lambda_\eps}\right).
$$
Therefore, by the change of variable $t=s/\lambda_\eps$, we obtain
\begin{align}\label{eq:energia}
\sqrt{\eps}\int_0^\ell\kappa_{\recovery}^2 \,ds
=
\frac{\sqrt{\eps}}{\lambda_\eps}
\int_0^{\ell(\tilde\recovery)}
\kappa_{\tilde\recovery}^2\,dt.
\end{align}

By construction
\begin{align}\label{eq:split}
\int_0^{\ell(\tilde\recovery)}\kappa_{\tilde\recovery}^2\,dt
&=
\int_{\hat\keycon}\kappa_{\hat\keycon}^2\,ds
+
\int_{\Sigma_\eps}\kappa_{\Sigma_\eps}^2\,ds
+
\int_{\gamma} \kappa_{\gamma}^2\,ds.
\end{align}
We start by noticing that
\begin{equation}\label{eq:alcune_curvature}
\kappa_{\rm int(\Sigma_\eps)}=0,
\end{equation}
where $\rm int(\Sigma_\eps)$ stands for the relative interior of $\Sigma_\eps$.
Moreover, from \eqref{eq:eps_k_squared_ricciolo} and \eqref{eq:energia_arco}, we get 
\begin{equation}\label{eq:altra_curvatura}
\int_0^{\ell(\hat\keycon)}\kappa_{\hat\keycon}^2\,ds
=
4\frac{\sqrt{2}}{\sqrt{\eps}}\tanh\!\left(\tfrac{\de}{\sqrt{2\eps}}\right)
+\frac{\theta_\eps}{r_\eps}.
\end{equation}
Therefore, plugging \eqref{eq:split} into \eqref{eq:energia}, we obtain
\begin{equation*}
\sqrt{\eps}\int_0^\ell\kappa_{\recovery}^2 \,ds
=
\frac{\sqrt{\eps}}{\lambda_\eps}
\left(
4\frac{\sqrt{2}}{\sqrt{\eps}}\tanh\!\left(\tfrac{\de}{\sqrt{2\eps}}\right)
+\frac{\theta_\eps}{r_\eps}
+\int_0^\ell \kappa_\gamma^2\,ds
\right).
\end{equation*}

Recalling that $\lambda_\eps\to1$, using \eqref{eq:scelta_de} and \eqref{eq:limIII}, and observing that
$$
\lim_{\eps\to0^+}\sqrt{\eps}\int_0^\ell \kappa_\gamma^2\,ds =0,
$$
we conclude that
\begin{equation}\label{eq:curv_limit_final}
\lim_{\eps\to0^+}
\sqrt{\eps}\int_0^\ell\kappa_{\recovery}^2 \,ds
=
4\sqrt{2}.
\end{equation}

On the other hand, by \eqref{eq:limite_uno}, \eqref{eq:limite_due} and \eqref{eq:limite_tre},
\begin{equation}\label{eq:tangent_limit_final}
\lim_{\eps\to 0^+}\frac1{2\eps^{1/2}}
\int_0^\ell |\partial_s\recovery-\partial_s\gamma|^2\,ds
= 4\sqrt2.
\end{equation}

We also claim that 
\begin{equation}\label{eq:recovery_limit_bis}
\lim_{\eps\to 0^+}
	\bigl\| (\kappa_{\recovery}-\kappa_\ga) - \omega \bigr\|_{\mathrm{flat}} = 0,
    \quad \text{ with } \quad \omega=2\pi \delta_{s_1}.
\end{equation}
Indeed, following the same strategy as in the compactness
Proposition~\ref{lem:compactness_closed_flat}, we fix $\radius>0$ and consider
 $(a_{\eps,\radius},b_{\eps,\radius})$, such that
$E_\radius^\eps
=
(a_{\eps,\radius},b_{\eps,\radius}).$
By construction of $\recovery$ and by the definition of $E_\radius^\eps$, it follows that
\begin{equation*}
a_{\eps,\radius}
\in
\big(
 \lambda_\eps s_1, \lambda_\eps(s_1+\ell(\hat\keycon))\big)
\end{equation*}
Now using \eqref{eq:aeps_to_a_closed}, \eqref{eq:lunghezza_keycon} and $\lambda_\eps\to 1$, we get that, for any fixed $\radius$, as $\eps \to 0^+$,
\begin{equation}\label{eq:ai_to_xi_closed}
a_{\eps,\radius} \to s_1.
\end{equation}
Combining \eqref{eq:flat_conv_closed} and \eqref{eq:ai_to_xi_closed}, we conclude that $0=\lim_{\eps\to 0+}\|(\kappa_{\gae}-\kappa_\ga)-2\pi\delta_{a_{\eps,\radius}}\|_{\mathrm{flat}}=\lim_{\eps\to 0+}\|(\kappa_{\gae}-\kappa_\ga)-2\pi\delta_{s_1}\|_{\mathrm{flat}}$, which is \eqref{eq:recovery_limit_bis}.

Combining  \eqref{eq:curv_limit_final},
\eqref{eq:tangent_limit_final} and \eqref{eq:recovery_limit_bis} we conclude that
$$\lim_{\eps\to0^+}\mathcal{G}_\eps(\recovery)=\sigma G(\ga,\omega).$$

\paragraph{The general case.}
We now extend the previous construction to the case of a finite number of curvature singularities. 
Let 
$$
0\leq s_1<\dots<s_N<\ell
$$
be points on the curve $\gamma:[0,\ell]\to\mathbb R^2$ at which the singularities are to be inserted. 

For each $i=1,\dots,N$, let $s_i'\in[0,\ell)$ be such that
$$
\partial_s\gamma(s_i')=-\partial_s\gamma(s_i),
$$
whose existence is guaranteed by Lemma~\ref{lem:opposite_tangent_point}.

For every $i=1,\dots,N$, let $\hat{\keycon}^i$ denote the rotated copy of the key construction inserted at $s_i$, and let $\Sigma^i_\eps$ denote the corresponding correcting segment inserted near $s_i'$. 
We assume that each $\hat{\keycon}^i$ and each $\Sigma^i_\eps$ is parametrized by arclength and satisfies the same matching conditions as in the case of a single singularity.

We collect all insertion points
$$
s_1,\dots,s_N,s_1',\dots,s_N'
$$
and reorder them increasingly along $[0,\ell)$, obtaining
$$
0\leq r_1\leq r_2\leq \dots\leq r_{2N}\leq \ell.
$$
For convenience, we also set $r_0:=0$ and $r_{2N+1}:=\ell$.

For each $j=1,\dots,2N$, we define the corresponding inserted piece $\upsilon^j_\eps$ by
$$
\upsilon^j_\eps=
\begin{cases}
\hat{\keycon}^i & \text{if } r_j=s_i \text{ for some } i\in\{1,\dots,N\},\\[4pt]
\Sigma^i_\epsilon & \text{if } r_j=s_i' \text{ for some } i\in\{1,\dots,N\}.
\end{cases}
$$
Moreover, for each $j=1,\dots,2N+1$, we denote by $\gamma^j$ the arc of $\gamma$ joining the endpoint of the $(j-1)$-st inserted piece to the initial point of the $j$-th inserted piece, with the convention that $\gamma^1$ starts at $\gamma(0)$ and $\gamma^{2N+1}$ ends at $\gamma(\ell)=\gamma(0)$.

With this notation, the preliminary recovery curve $\tilde\gamma_\eps$ is defined by concatenation as
$$
\tilde\gamma_\eps
=
\gamma^1
\ast
\upsilon^1_\eps
\ast
\gamma^2
\ast
\upsilon^2_\eps
\ast
\cdots
\ast
\gamma^{2N}
\ast
\upsilon^{2N}_\eps
\ast
\gamma^{2N+1}.
$$
Here the symbol $\ast$ denotes the usual concatenation of curves, after reparametrization by arclength on consecutive intervals.

Finally, as in the case of one singularity, we define the recovery sequence $\gamma_\eps$ by rescaling $\tilde\gamma_\eps$ through the homothety
$$
\lambda_\eps:=\frac{\ell}{\ell(\tilde\gamma_\eps)},
\qquad
\gamma_\eps(s):=
\lambda_\eps\,\tilde\gamma_\eps\!\left(\frac{s}{\lambda_\eps}\right),
\qquad s\in[0,\ell].
$$
Then $\gamma_\eps$ is a closed curve in $H^2([0,\ell];\mathbb R^2)$ satisfying $\ell(\gamma_\eps)=\ell$.

The estimates obtained in the case of just one curvature singularity extend to the present setting. Arguing exactly as above for each singularity, we deduce that
$$
\lim_{\eps\to0^+} \mathcal{G}_\eps(\gamma_\eps)
=
\sigma\, G(\gamma,\omega), \quad \text{ where } \omega=2\pi\sum_{j=1}^N c_j\delta_{s_j}, \, c_j\in\mathbb Z,\, 0\le s_1<\dots<s_N< \ell.
$$
This concludes the proof.


\section*{\texorpdfstring{\normalsize Acknowledgements.}{Acknowledgements.}}
\vspace{-0.8em}
All authors are members of the Gruppo Nazionale 
per l’Analisi Matematica, la Probabilità e le loro 
Applicazioni (GNAMPA) of the Istituto Nazionale di Alta Matematica (INdAM).
\bibliographystyle{plain}
\bibliography{refs_exp_gamma_elastiche}
\nada{
}
\end{document}